\documentclass[12pt]{article}	
\usepackage{graphicx}
\usepackage[utf8]{inputenc}
\usepackage{indentfirst,csquotes}

\usepackage[margin=2.5cm]{geometry}
\usepackage[english]{babel}
\usepackage{tikz}
\usepackage{fancyhdr} 
\usepackage{amsthm}
\usepackage{nicematrix}
\usepackage{xcolor} 
\usepackage{bbm}

\usepackage{mathrsfs}
\usepackage{hyperref}
\usepackage{amssymb}
\usepackage{mathtools}
\usepackage{amsmath}
\usepackage{latexsym}
\usepackage{esint}
\usepackage{wasysym}
\usepackage{pst-eucl}
\usepackage{comment}
\usepackage{orcidlink}
\usepackage{hyperref}
\usepackage{blkarray}
\usepackage{soul}
\hypersetup{
    colorlinks=true,
    linkcolor=blue,
    citecolor=teal
    }

\let\P\undefined 
\newcommand{\norm}[1]{\left\lVert#1\right\rVert}
\newcommand{\nor}[1]{\left\lvert#1\right\rvert}

\makeatletter
\newenvironment{customproof}[1]{%
  \par\pushQED{\qed}%
  \normalfont\topsep6pt \trivlist
  \item[\hskip\labelsep\itshape
    Proof of #1\@addpunct{.}]\ignorespaces
}{%
  \popQED\endtrivlist\@endpefalse
}
\makeatother

\newtheorem{Theorem}{Theorem}
\numberwithin{Theorem}{section}
\newtheorem{Proposition}[Theorem]{Proposition}
\newtheorem{Lemma}[Theorem]{Lemma}
\newtheorem{Corollary}[Theorem]{Corollary}
\theoremstyle{definition}
\newtheorem{Definition}[Theorem]{Definition}
\theoremstyle{remark}
\newtheorem{Remark}[Theorem]{Remark}
\numberwithin{equation}{section}

 \DeclareMathOperator{\C}{\mathcal{C}}
 \DeclareMathOperator{\R}{\mathbb{R}}

\DeclareMathOperator*{\supp}{supp}

\DeclareMathOperator*{\tr}{tr}
\newcommand*{\Mp}{\mathcal{M}_{K}^{+}}
\newcommand{\Mm}{\mathcal{M}_{K}^{-}}
\newcommand{\p}{\partial}

\newcommand{\osc}{\operatorname{osc}}

\newcommand{\P}{\mathcal{P}}
\newcommand{\eps}{\varepsilon}

\usepackage{dirtytalk}

\begin{document}
\title{A transmission problem arising from the two-phase Stefan problem}

\author{Fausto Ferrari$^{\orcidlink{0000-0002-2419-7018}}$
  \and
  Davide Giovagnoli$^{\orcidlink{0009-0005-3375-5199}}$
  \and 
  David Jesus$^{\orcidlink{0000-0001-7065-4737}}$
}

\author{Fausto Ferrari$^{\orcidlink{0000-0002-2419-7018}}$
\and 
Davide Giovagnoli$^{\orcidlink{0009-0005-3375-5199}}$
\and
David Jesus$^{\orcidlink{0000-0001-7065-4737}}$
}
\newcommand{\Addresses}{{
  \bigskip
  \footnotesize

  Fausto Ferrari, \textsc{Dipartimento di Matematica,
Universit\`a di Bologna\\ 
Piazza di Porta San Donato 5, 40126 Bologna, Italy}\par\nopagebreak
  \textit{E-mail address:} {\tt fausto.ferrari@unibo.it}\\

  \medskip

    Davide Giovagnoli, \textsc{Dipartimento di Matematica,
Universit\`a di Bologna\\ 
Piazza di Porta San Donato 5, 40126 Bologna, Italy}\par\nopagebreak
  \textit{E-mail address:} {\tt d.giovagnoli@unibo.it}\\

  \medskip

     David Jesus, \textsc{Applied Mathematics and Computational Sciences (AMCS), King
Abdullah University of Science and Technology (KAUST),  Thuwal 23955-6900, Kingdom of Saudi Arabia}\par\nopagebreak
  \textit{E-mail address:} {\tt djbj1993@gmail.com} \\
 }
}

\date{}

\maketitle

\noindent{\bf Abstract.}
We study a parabolic transmission problem whose interface condition depends on the time derivative of the solution. This model arises naturally as the linearized limiting profile of a two-phase inhomogeneous Stefan problem after the hodograph transform. Our main result shows the existence and uniqueness of a classical solution. We also prove a Harnack inequality for a general class of transmission problems and establish \(C^{1,\alpha}\) estimates up to the interface from both sides. These results can be used to obtain \(C^{1,\alpha}\) regularity of flat free boundaries for the two-phase Stefan problem with distributed sources. 



\section{Introduction}

In this paper, we investigate the regularity of solutions to a  parabolic transmission problem which arises as the linearized limiting profile of the two-phase Stefan problem
\begin{equation} \label{eq:main1}
\begin{cases}
        \lambda \p_t u - \tr(A^{\pm}(t) D^2 u)= 0 & \text{ in }\C_{1}^{\pm}:= (Q_1 \cap \{ \pm \,x_n> 0\}) \times (-1,0] \\
        \p_t u= -\hat{\gamma}(t) \cdot  \nabla' u  + \gamma^+( t) \p_n u^+ - \gamma^-(t) \p_n u^- & \text{ on }\C_{1}':= (Q_1 \cap \{ x_n= 0\}) \times (-1,0].
\end{cases}
\end{equation}

Here, $Q_1$ denotes the unit cube in $\mathbb{R}^n$, and $\lambda > 0$ is a parameter stemming from the lack of scaling of the Stefan problem.  Both the symmetric $n \times n$ matrices $A^\pm$ and the transmission coefficients $\hat{\gamma}$ and $\gamma^\pm$ depend only on time. Throughout the paper, $\nabla' u$ denotes the spatial gradient with respect to the first $n-1$ variables, whereas $\partial_n u^\pm$ refer to the normal spatial derivatives of $u^\pm := u\vert{}_{\overline{\mathcal{C}_1^\pm}}$, respectively.

In addition to its physical motivation from the Stefan problem, the main  novelty of this paper lies in the transmission condition. The interface condition explicitly involves the time derivative $\partial_t u$ coupled with the jump of the normal derivatives. This specific feature introduces significant analytical challenges in the study of parabolic transmission problems. 

The regularity of solutions to the transmission problem \eqref{eq:main1} plays a key role in establishing the free boundary regularity for nondegenerate flat solutions to the two-phase inhomogeneous Stefan problem. We make this connection precise in Section \ref{sec:motivation}, where we derive \eqref{eq:main1} as the linearized limiting profile associated with the Stefan problem in the nondegenerate regime.

It is worth mentioning that, in \cite[Open problem 7]{salsa2012two}, Salsa remarked that the linearization arguments developed by De Silva in \cite{DeSilva2011} to treat inhomogeneous elliptic one-phase problems, recently extended to the parabolic one-phase case in \cite{DFS23}, could potentially be applied to the two-phase Stefan problem with a source term. This intuition is partly confirmed by the results of this paper.

Transmission problems arise in the modeling of physical phenomena exhibiting abrupt changes across a fixed interface. 
Since Picone’s pioneering work in elasticity \cite{picone1954probleme},
several now-classical results helped establish the elliptic theory.
Among them, for instance, we recall \cite{Campanato,Lions,Schechter,Stampacchia}
and for a comprehensive monograph we refer to \cite{Borsuk}.
For some  recent advances, we mention
\cite{LiVogelius,LiNirenberg,CittiFerrari,DFSFully,desilva2018,CSS_Trans,soria2023regularity,GJ_transmission,erneta2026regularity}.

In contrast, the parabolic counterpart remains much less explored, although a few recent contributions can be found in \cite{jesus2026fully,KS}. The results in this paper represent a further step in the understanding of parabolic transmission problems.  We also refer to \cite{bianca2022transmission,Pini2025Transmission} for an overview of regularity results for both elliptic and parabolic transmission problems and their connection with free boundary problems.

To properly frame our contribution, we  introduce the structural hypotheses on the coefficients of \eqref{eq:main1}. While these conditions might initially appear non-standard, they are intrinsically dictated by the scaling properties of the original physical model as shown in Section \ref{sec:motivation}. 

Let $ \lambda \in (0,\lambda_0]$ for some universal constant $\lambda_0$, and fix $K\geq1$.
We assume that the transmission coefficients $\gamma^\pm, \hat{\gamma}$ and the matrices $A^{\pm}$ belong to $C_t^{0,1}([-1,0])$ and satisfy, for any $t \in (-1,0]$, the following bounds
\begin{equation}
\begin{aligned} \label{Assump1}
&K^{-1} I_n\leq  A^{\pm}(t) \leq K I_n, \quad |{\p_t} A^{\pm}(t)| \leq K \lambda^{-1},\\
&|\hat \gamma(t)| \leq K, \quad K^{-1}\leq \gamma^{\pm}(t) \leq K, \\
&|{\p_t}\hat \gamma(t)| \leq \lambda^{-1}, \quad | {\p_t}\gamma^{\pm}(t)| \leq \lambda^{-1}.
\end{aligned}
\end{equation}

We now present our main results for the transmission problem \eqref{eq:main1}. In Theorem \ref{Thm:C1gamma}, we show that there exists a unique classical solution satisfying gradient and Hessian estimates uniform in $\lambda$. This result also provides a $C^{1,\alpha}$-type expansion at the interface by affine functions whose coefficients satisfy the transmission condition.

\begin{Theorem} \label{Thm:C1gamma}
Let $g$ be a continuous function on $\p_p \C_1$ with $\| g\|_{L^\infty(\p_p \C_1)}\leq 1$. Then there
exists a unique classical solution $u$  to \eqref{eq:main1} with coefficients satisfying \eqref{Assump1} such that $u\in C(\overline{\C_1})$ and $u=g$ on $\p_p \C_1$.
  Moreover, for some positive universal constants $C$ and $\alpha$, the solution $u$ satisfies
    \begin{align} \label{eq:C11}
        |\nabla u|\leq C, \qquad |D^2 u| \leq C \quad \text{ in }\C_{1/2}^\pm:= (Q_{1/2} \cap \{\pm \,x_n >0\}) \times (-1/2,0], 
    \end{align}
    and for any $r \leq \frac{1}{2}$ there exist two functions $\ell_{a^+,b}$ and $\ell_{a^-, b}$   affine in $x$ for each fixed $t$ of the form
    \[
    \ell_{ a^\pm, b}(x,t)=a^{\pm}(t) \cdot x + b(t),
    \] 
    such that
    \begin{equation}\label{eq:transmission_1}
        \| u - \ell_{a^{\pm},b} \|_{L^{\infty}(\C_r^\pm)} \leq C r^{1+{ \alpha}},
    \end{equation}
    where $a^\pm(t)=(\hat a, a_n^\pm(t))\in \mathbb{R}^{n-1}\times\mathbb{R}$ with $|\hat a| \leq C$, $|\p_t a_n^{\pm}| \leq C r^{{ \alpha}-1} \lambda^{-1}$ and
    \[ \p_t b(t) = -\hat{\gamma}(t) \cdot \hat a  + \gamma^+(t) a_n^+(t) - \gamma^-(t) a_n^-(t).\]
\end{Theorem}
 The estimate \eqref{eq:transmission_1} can be applied to establish regularity of the free boundary for the two-phase inhomogeneous Stefan problem, using the techniques developed in \cite{DFS23} for the homogeneous one-phase problem, and subsequently extended to the inhomogeneous case in \cite{ferrari2026free}. 

A fundamental tool in the proof of Theorem \ref{Thm:C1gamma} is the following Harnack inequality, which 
holds for a more general class of transmission problems denoted by $S^*_{\lambda,T}(f^{\pm})$, that incorporates the model \eqref{eq:main1}. For the precise definition of $S^*_{\lambda,T}(f^{\pm})$, see Section \ref{sec:preli}.

\begin{Theorem}\label{Lm:Harnack}
    Let $u \in S^*_{\lambda,T}(f^{\pm})$ in $\C_1$. Then there exist positive universal constants $c < 1$ and $C$ such that
    \[
    \osc_{\C_{1/2}} u \leq (1-c) \osc_{\C_1} u + C \lambda^{-1}\|f^\pm\|_{L^\infty(\C_1)}.
    \]
\end{Theorem}

In the one-phase Stefan problem, an estimate of this type yields an \textit{almost H\"older} modulus of continuity, which provides the compactness needed for the corresponding sequence of error functions; see \cite[Theorem 5.4]{DFS23} and \cite[Proposition 4.8]{ferrari2026free}. The result above, however, is not a direct application of these one-phase arguments. Its proof requires genuinely new two-phase barriers that transfer oscillation information across the interface while respecting the transmission condition.

Iterating Theorem \ref{Lm:Harnack}, we obtain uniform H\"older estimates with respect to an appropriate distance $d_\lambda$. This distance is consistent with the scaling of the problem, interpolating between the standard parabolic distance away from the interface and the hyperbolic distance close to $\{ x_n=0 \}$. The precise definition of $d_\lambda$, together with the corresponding H\"older continuity result, is provided in  Proposition \ref{Prop:holder}.

We conclude this introduction by briefly recalling the Stefan problem.
Let $T>0$ and let $\Omega$ be a bounded domain in $\R^n$, the two-phase Stefan problem can be modeled as
\begin{equation}\label{eq:twophaseStef}
 \begin{cases}
     H^+\partial_tu - k^{+}\Delta u=f& \quad \text{ in } \Omega_T^+(u):=\left( \Omega \times (0,T] \right) \cap \{u>0\},\\ H^{-}\partial_tu -  k^{-}\Delta u =f& \quad \text{ in }\Omega_T^-(u):=\left( \Omega \times (0,T] \right) \cap \{u\leq 0\}^0,  \\
     \displaystyle \frac{ \partial_{t} u^+}{\nor{\nabla u^+}} = k^{+} \nor{\nabla u^+}- k^{-}\nor{\nabla u^-} = -\frac{ \partial_{t} u^-}{\nor{\nabla u^-}} & \quad \text{on}\:\: \mathcal{F}(u):=\left(\Omega \times (0,T] \right) \cap \partial\{ u > 0  \},
 \end{cases}
 \end{equation} 
 where $H^+,H^-,k^{+},k^{-}>0$ denote, respectively,  the heat capacities  and the thermal conductivity of the two states. These quantities are related through the thermal diffusivities $D^{+},D^{-}$ of the two states, defined by
 \[
 D^{\pm} = \frac{k^{\pm}}{H^{\pm}}.
 \]

Here we assume that $H^\pm$ and $k^\pm$ are constant and independent of the temperature. This is the classical formulation of the two-phase Stefan problem and is the setting considered in the regularity theory developed by Athanasopoulos, Caffarelli, and Salsa in \cite{ACS,ACS2,ACS3}; see also \cite[Chapter XI]{carslaw1959conduction} for the physical background of the model.
For a comprehensive treatment of the various formulations of the Stefan problem, we refer the reader to \cite{vis,Friedman,Friedmanerr}.

Following the strategy introduced by De Silva in \cite{DeSilva2011} for elliptic equations and later adapted to the Stefan problem in the perturbative framework of \cite{DFS23}, we focus on the transmission problem \eqref{eq:main1} and its associated extremal class. In fact, this problem naturally arises as the linearized model in the improvement-of-flatness scheme, and the regularity results established in this paper provide the analytical input needed to study the free boundary of the two-phase Stefan problem \eqref{eq:twophaseStef}. The exact derivation of \eqref{eq:main1} from \eqref{eq:twophaseStef} is detailed in Section \ref{sec:motivation}.

Recent works \cite{ferrari2026free,ferrari2025geometry} have successfully implemented this strategy for the one-phase Stefan problem with distributed sources. For the homogeneous two-phase Stefan problem, we refer instead to the trilogy \cite{ACS,ACS2,ACS3}, whose approach differs substantially and is rooted in the elliptic theory developed in \cite{CS}.

The proof of Theorem \ref{Lm:Harnack} relies on the construction of suitable barriers that allow us to transfer  information about oscillation decay across the interface. This construction is carried out in Section \ref{sec:harnack}.
Subsequently, to obtain higher regularity in the proof of Theorem \ref{Thm:C1gamma}, we reduce the analysis to a one-dimensional problem along the direction perpendicular to the interface, see Section \ref{sec:C1gamma}.

\smallskip
\noindent\textbf{Plan of the paper.} The paper is organized as follows.  In Section \ref{sec:motivation}, we motivate our study by detailing the connection between the classical two-phase Stefan problem and our limiting transmission model via the hodograph transform.
Then, in Section \ref{sec:preli}, we collect some preliminary results, introduce notation, and define the notion of viscosity solution. Section \ref{sec:harnack} is devoted to the proof of the Harnack inequality of Theorem \ref{Lm:Harnack}, which leads to the H\"older continuity estimates established in Section \ref{sec:holder}. The comparison principle for viscosity solutions is proved in Section \ref{sec:compare}. Section \ref{sec:existence} is devoted to the existence of the solution of \eqref{eq:main1}. In Section \ref{sec:C1gamma}, we establish Theorem \ref{Thm:C1gamma} via a reduction to a one-dimensional problem. Finally, Section \ref{sec:1dcase} provides the technical analysis and Schauder estimates for the reduced one-dimensional problem.

\section{Motivation of the problem} \label{sec:motivation}

In this section we explain the role of the transmission problem
\eqref{eq:main1} 
as the natural linearized model associated with the
two-phase Stefan problem with distributed sources. We argue formally, assuming that the solution is classical, and we focus on the nondegenerate
flat regime in which both phases contribute to the motion of the free
boundary. Starting from a flat solution of \eqref{eq:twophaseStef}, we
apply the hodograph transform to fix the free boundary, and then
linearize the resulting nonlinear transmission problem around the
corresponding planar profile. This procedure yields precisely the limiting equation \eqref{eq:main1}.

In a $\lambda$-neighborhood of a regular free boundary point, a
nondegenerate classical solution is expected to be of size $\lambda$ and
to be well approximated by a two-plane profile, whose slopes encode the
incoming heat fluxes from the two phases. We therefore impose the
following flatness hypothesis.

Let us fix a large constant $K>1,$ possibly depending on $H^{+}$ and $H^{-}$. Let $0<\lambda \leq \lambda_0$ for some $\lambda_0$ small and let $u$ be a solution to \eqref{eq:twophaseStef} in $B_{\lambda} \times(-\lambda,0]$. We will make the following assumption, which we will refer to as $\varepsilon$-flatness.
Assume that $0\in \p \{u>0\},$ and there exist smooth functions $a_n^+, a_n^-$ and $b$  such that  
\begin{equation} \label{eps-flat}
\left|u - a_n^+(t) \, \left (x_n - b(t)\right)_+ + a_n^-(t) \, \left (x_n - b(t)\right)_- \right| \leq C\varepsilon \lambda \quad \text{in } B_\lambda \times [-\lambda, 0],
\end{equation}
 with 
 \begin{equation} \label{eq:nondeghp}
 K^{-1}\le a_n^\pm \le K, \quad \quad |{\p_t} a_n^{\pm}(t)| \le  c_0\lambda^{-2}, \quad \quad {\p_t} b(t)= k^{-}a_n^-(t)-k^{+}a_n^+(t),\quad\quad  \|f\|_{L^\infty}\leq \varepsilon^2,
 \end{equation}
for some small positive $\varepsilon$ depending only on $K $ and $n$.

We remark that the free boundary is invariant under hyperbolic dilation, which consequently sets the correct scale in its neighborhood. 
Notice that the above conditions provide a characterization of the flatness of the solution.
Furthermore, if $f \equiv 0$, the condition on $u$ can be  interpreted as the case in which $u$ is an $o(\lambda)$ perturbation of a two-phase traveling wave solution of the form
\begin{equation}
    u_0(x_n,t) = \frac{1}{H^+-H^-}\left[ \left(\exp\left(\frac{ax_n + a^2 t}{D^+}\right) -1\right)_+ -\left(\exp \left( \frac{a x_n + a^2 t}{D^{-}} \right) -1\right)_-  \right].
\end{equation}
Here, we assume  $H^+>H^{-}$. 
 In this case, the functions $a_n^+$, $a_n^-$ and $b$ are defined as
 \[
 a_n^+(t)= \frac{a}{D^+(H^+-H^-)}, \; a_n^{-}(t)=\frac{a}{D^-(H^+-H^-)} \in (K^{-1},K) \quad \text{and $b(t)=-at$}.
 \]
The relation ${\p_t} b(t)=k^{-}a_n^-(t)-k^{+}a_n^+(t)$ implies that the approximating linear functions in $x$ satisfy the free boundary condition. We will refer to this assumption as the $\eps$-flatness of the solution.

\begin{Remark}

The $\eps$-flatness condition on $u$ is restrictive compared to the general situations that may occur. As in the homogeneous case, we cannot expect flat free boundaries to regularize instantaneously. As shown in \cite[Section 8.3]{CS} and \cite[Section 10]{ACS2}, there are explicit counterexamples of persistent Lipschitz corners when the heat fluxes from both sides vanish. To prevent this phenomenon, it is natural to prescribe a nondegeneracy condition ensuring that at least one of the two heat fluxes does not vanish. The $\eps$-flatness of $u$ describes the setting where both heat fluxes are nonvanishing and of comparable size. Whenever one of the two phases is small but not negligible, the solution is expected to behave like a one-phase solution.
In this paper, we do not aim to provide a general characterization of flat free boundaries for the two-phase problem \eqref{eq:twophaseStef}, rather, we focus  on the nondegenerate case starting from the $\eps$-flatness condition. A full characterization remains a very interesting open problem.

\end{Remark}

To show the connection between the Stefan problem \eqref{eq:twophaseStef} and the parabolic transmission problem \eqref{eq:main1}, we employ the hodograph transform, following the strategy introduced in \cite{DFS23}.
 To avoid confusion, from now on, we will use $u_{S}$ to denote a solution to the Stefan problem \eqref{eq:twophaseStef} and $u_H$ its Hodograph transform.

Geometrically, we can view the graph of $u_S$ in $\mathbb{R}^{n+2}$, defined by
$$
\Gamma := \{ (x,x_{n+1},t) \mid x_{n+1}= u_S(x_1, x_2, \dots, x_n,t) \},
$$
as the graph of a possibly multi-valued function $u_H$ with respect to the $x_n$-coordinate
$$
\Gamma := \{ (x,x_{n+1},t) \mid x_n= u_H(x_1, x_2, \dots, x_{n-1}, x_{n+1}, t) \}.
$$
By deriving the equations satisfied by this transformed function in the new spatial coordinates $y = (x_1, \dots, x_{n-1}, x_{n+1})$, mimicking the computations and notation from \cite[Section 3]{ferrari2026free}, we obtain
\begin{align*}
 &\nabla u_S = - \frac{1}{\p_n u_H} \left( \nabla' u_H, -1 \right), \quad  \p_t u_S= - \frac{\p_t u_H}{\p_n u_H} , \\
 &D^{2}u_S= -\frac{1}{\partial_{n}u_H} A^T(\nabla u_H) \, D^2 u_H \, A(\nabla u_H).
 \end{align*}
 Here, $A(\nabla u_H)$ is a square matrix that coincides with the identity matrix everywhere except on the $n$-th row, where the entries are given by the right-hand side of $\nabla u_S$ above.

The advantage of this change of variables is that the unknown free boundary of \(u_S\) is mapped to the fixed interface
\(\{y_n=0\}\), while the Stefan condition becomes a nonlinear
transmission condition.
 
Consequently, we can transform the Stefan problem \eqref{eq:twophaseStef} into a parabolic transmission problem for $u_H$, coupled with an oblique derivative condition at the fixed interface $\{ y_n=0 \}$. This yields
\begin{equation}\label{eq:HodographStef}
        \begin{cases}
        \partial_{t} u_H - \tr( \Bar{A}^{\pm}(\nabla u_H) D^2 u_H) + (H^{\pm})^{-1} \partial_{n}u_H \,  f(y',u_H,t)=0 & \text{in} \,\, \{ \pm \, y_n > 0 \},\\
~
         \\
       \displaystyle \frac{\partial_{t} u_H^+}{ \nor{(\nabla' u_H^+, -1)}} = -k^{+}\frac{\nor{(\nabla' u_H^+, -1)}}{\partial_{n}u_H^+} + k^{-}\frac{\nor{(\nabla' u_H^-, -1)}}{\partial_{n}u_H^-} = \frac{\partial_{t} u_H^-}{ \nor{(\nabla' u_H^-, -1)}}& \text{on} \,\,  \,\, \{ y_n = 0 \},
    \end{cases}
    \end{equation}
    where, taking $p=(p',p_n) \in \R^n$, we can express $A$, as the following positive definite matrix
    \[
\Bar{A}^{\pm}(p)= D^{\pm} \cdot \begin{pNiceArray}{c|c}
\text{ \begin{Large} $I_{n-1}$ \end{Large}} &   \displaystyle -\frac{p'}{p_n}\\
  \hline
 \displaystyle -\frac{p'}{p_n} &  \displaystyle \frac{1}{p_n^2} \left( 1+\nor{p'}^2 \right)
\end{pNiceArray}.
\]

Since $u_H$ is continuous across the fixed interface $\{y_n=0\}$, its traces from both sides coincide. Consequently, the tangential derivatives of $u_H^\pm$ match, i.e., $\nabla' u_H^+ = \nabla' u_H^-$. Thus, \eqref{eq:HodographStef} reduces to
 
\begin{equation}\label{eq:ReduHodoStef}
        \begin{cases}
        \partial_{t} u_H - \tr( \Bar{A}^{\pm}(\nabla u_H) D^2 u_H) + (H^{\pm})^{-1} \partial_{n}u_H \,  f(y',u_H,t)=0 & \text{in} \,\, \{ \pm \, y_n > 0 \},\\
~
         \\
       \displaystyle \partial_{t} u_H^+ =g(\nabla' u, \p_n u_H^+, \p_n u_H^- )= \partial_{t} u_H^-& \text{on} \,\,  \,\, \{ y_n = 0 \},
    \end{cases}
    \end{equation}
where  
\begin{equation}\label{eq:def_g}
g(p', p_n^+, p_n^- )= -\left(1 +  |p'|^2 \right) \left(\frac{k^{+}}{p_n^+} - \frac{k^{-}}{p_n^-} \right).
    \end{equation}

    Defining the set 
\begin{equation*}
    \mathcal{R}_K:= B_{K}(0) \cap \{ p_n  \geq K^{-1}\} \subset \mathbb{R}^{n},
\end{equation*}
we notice that if $(p',p_n^\pm)\in \mathcal{R}_K$, then $\partial_{p_{n}^+} g(p',p_n^+,p_n^-) > 0$ and $\partial_{p_{n}^-} g(p',p_n^+,p_n^-) < 0$. 
By choosing $K > 1$ large enough, we may assume that $\tr(\Bar{A}^{\pm}(p',p_n^\pm)M)$ is uniformly elliptic in $M$ with ellipticity constants $K^{-1}$ and $K$ for every $(p',p_n^\pm) \in \mathcal{R}_K$, see \cite[Section 3]{ferrari2026free}.
Furthermore, up to enlarging $K$ if necessary, we can assume the following bounds inside $\mathcal{R}_K$
\begin{equation}\label{eq:boundg}
\begin{aligned}
        &\nor{\nabla_{p} \tr(\Bar{A}(p',p_n^\pm)M)} \leq K \nor{M}, \quad \norm{g}_{C^1} \leq K,  \\   & \partial_{p_{n}^+} g(p',p_n^+,p_n^-)  \geq K^{-1}, \quad \partial_{p_{n}^-} g(p',p_n^+,p_n^-)  \leq - K^{-1}.
\end{aligned}
\end{equation}

\begin{Remark}
    \label{Rmk:flat}
    Using the flatness assumption on the free boundary given by \eqref{eps-flat}, we can localize the set where $u_H$ is potentially multi-valued.  Specifically, one can  prove that $u_H$ is single-valued in the set $\{|y_n| \geq C_0\varepsilon\lambda\}$ and $\nabla u_H \in \mathcal{R}_{K}$.  
    
    Let $t\in [t_0-\lambda^2,t_0+\lambda^2]$, then using the bounds for $|{\p_t}a|$ and $|{\p_t}b|$ in the assumptions \ref{eq:nondeghp}, for any $t_1, t_2 \in (t_0,t)$ we get
    \begin{equation*}
        \begin{aligned}
            &\nor{a_{n}^\pm(t)-a_{n}^\pm(t_0)} \leq \nor{{\p_t}a_{n}^\pm(t_1)} \nor{t-t_0} \leq c_0  \\
            &\nor{b(t)-b(t_0)} \leq \nor{a_n^-(t_2) -a_{n}^+(t_2)} \nor{t-t_0}  \leq 2K\lambda^{2}.
        \end{aligned}
    \end{equation*}
Now following the reasoning of \cite{DFS23}, as in \cite[Remark 3.3]{ferrari2026free},  we can analyze the sets $\{y_n \geq C_0 \varepsilon \lambda\}$ and $\{y_n \leq -C_0 \varepsilon \lambda\}$ separately to reach the desired conclusion.

By applying the inverse function theorem, we can reformulate the flatness hypotheses in terms of the hodograph transform as
 \begin{align*}
        &\nor{u_H - (\bar a_n^\pm (t)y_n+\bar b(t))} \leq C\varepsilon \bar \lambda,\quad   \textit{ in } \mathcal{C}^\pm_{\bar \lambda},  \\
        &{\p_t} \bar b(t)=g(0, \bar a_n^+(t),  \bar a_n^-(t)), \quad  K^{-1} \leq \bar a_n^\pm \leq K, \quad \nor{ {\p_t} \bar a_n^\pm} \leq c_1 \bar \lambda^{-2},
    \end{align*}
    with $\bar \lambda\leq \lambda_1$ for some $ \lambda_1>0$ small. 
\end{Remark}

Setting aside the potential multi-valuedness of $u_H$ for the moment, we assume $\partial_{n}u_H^\pm > K^{-1}$, which is guaranteed to hold at least away from the interface as justified by Remark \ref{Rmk:flat}.
For simplicity of notation, from now on we will denote the transformed function $u_H$ simply as $u$, while the original solution to \eqref{eq:twophaseStef} will be denoted by $u_S$. We emphasize that $u$ inherits its core properties from being the transform of $u_S$. Additionally, we drop the bar from $\bar A^{\pm}, \bar \lambda, \bar a_{n}^\pm, \bar b$ and use the variable $x$ instead of $y$.

We now focus on linearizing the equation around the approximating functions $l_{a,b}^{\pm}$. For each fixed $t$, $l_{a,b}^{\pm}$ are linear in $x$, their coefficients in the $x'$ variables are independent of $t$, and they satisfy the transmission condition in \eqref{eq:ReduHodoStef}. Specifically we have

\[
l_{a,b}^{\pm}(x,t):= a^{\pm}(t) \cdot x + b(t)
\]
with 
\begin{equation} \label{Hp:condit-assum-flat}
a^{\pm}(t)=(\hat a, a_n^{\pm}(t)), \quad \hat a \in \mathbb{R}^{n-1}, \quad a^{\pm}(t) \in \mathcal{R}_{K},\quad \nor{{\p_t}a_n^{\pm}(t)} \leq \delta \varepsilon \lambda^{-2},
\end{equation}
 where
    \begin{equation*}
        \varepsilon \leq \varepsilon_{1}, \quad \lambda \leq \lambda_{1}, \quad \lambda \leq \delta \varepsilon,
    \end{equation*}
for some $\delta>0$ and such that 
\[
{\p_t}b(t)=g(\hat a, a_n^{+}(t),a_n^{-}(t)).
\]

To establish the $C^{1,\alpha}$ regularity of the free boundary, we rely on the strategy developed in \cite{DFS23} for the one-phase Stefan problem, which was originally introduced by De Silva \cite{DeSilva2011} for the elliptic case. This approach consists of obtaining an improvement of flatness for solutions to \eqref{eq:twophaseStef} that satisfy the flatness condition \eqref{eps-flat} under the assumptions \eqref{eq:nondeghp}.
These solutions inherit the regularity of the solutions to the limiting transmission problem \eqref{eq:main1}. 

The main idea is to linearize our problem near $l_{a,b}^{\pm}$.  For this, we introduce the error function $w$, defined as
\[
u(x,t)= l_{a,b}^{\pm}(x,t) + \varepsilon \lambda w\left( \frac{x}{\lambda},\frac{t}{\lambda}\right) \quad \text{for }(x,t) \in \C_{\lambda}^{\pm}:= Q_\lambda^{\pm} \times (-\lambda,0].
\]
The function $w$ is defined in $\C_1$ and, due to the flatness assumption, satisfies $|w| \leq 1$.
Since $u$ solves \eqref{eq:ReduHodoStef}, inserting the definition of $l_{a,b}^{\pm}$ reveals that $w$ formally satisfies
\begin{equation}\label{eq:LinTwoStefan}
        \begin{cases}
        \lambda {\p_t} a_n^{\pm} (\lambda t) x_n + {\p_t} b(\lambda t) + \varepsilon \partial_t w(x,t) - \tr\big(A^{\pm}(a^{\pm}(\lambda t) + \varepsilon \nabla w) \frac{\varepsilon}{ \lambda} D^2 w )\big) \\  +(H^{\pm})^{-1}\left( \varepsilon\partial_{n}w+ a_n^{\pm}(\lambda t) \right) f(\lambda x', l_{a,b}^{\pm} + \varepsilon\lambda w, \lambda t) =0  & \text{in } \mathcal{C}_1^\pm,\\
        {\p_t} b(\lambda t) +  \varepsilon \partial_{t} w =  g(\hat a + \varepsilon \nabla' w, a_n^{+}(\lambda t) + \varepsilon \p_n w^+,a_n^{-}(\lambda t) + \varepsilon \p_n w^-)& \text{on } \mathcal{C}'_1.
    \end{cases}
    \end{equation}
    By multiplying the first equation in \eqref{eq:LinTwoStefan} by $\lambda \varepsilon^{-1}$ and the second by $\varepsilon^{-1}$, and then taking the limits as $\varepsilon \to 0$ and $\delta \to 0$, we arrive at the limiting transmission problem
\begin{equation}\label{eq:LimitTwoStefan}
        \begin{cases}
        \lambda \partial_t v = \tr(A_{\lambda}^{\pm}(t) D^2 v )& \text{in} \,\, \mathcal{C}_1^{\pm},\\
     \partial_t v = \nabla g (\hat a, a_n^+(\lambda t),a_n^-(\lambda t)) \cdot (\nabla' v, \p_n v^+, \p_n v^{-}) & \text{on} \,\,  \,\, \mathcal{C}'_1,
    \end{cases}
    \end{equation}
    where  $A_{\lambda}^{\pm}(t):= A^{\pm}(a(\lambda t))$.
By explicitly expressing the transmission condition using the definition of $g$ in \eqref{eq:def_g}, we obtain
\begin{align*}
\p_t v= -2 \left( \frac{k^+}{a_n^+(\lambda t)}- \frac{k^-}{a_n^-(\lambda t)}\right) \hat a\cdot \nabla'v + \left(1+|\hat a|^2 \right) \left( \frac{k^+ \p_n v^+}{(a_n^+(\lambda t))^2} -\frac{k^- \p_n v^-}{(a_n^-(\lambda t))^2}\right).
\end{align*}
At this stage, we can define the time-dependent coefficients as 
\[
\hat \gamma_\lambda(t) := 2 \left( \frac{k^+}{a_n^+(\lambda t)}- \frac{k^-}{a_n^-(\lambda t)}\right)\hat a, \qquad \gamma_\lambda^{\pm}(t):=\left(1+|\hat a|^2 \right) \frac{k^\pm}{(a_n^\pm(\lambda t))^2}.
\]
With this notation, the interface condition becomes $\partial_t v = -\hat{\gamma}_\lambda(t) \cdot \nabla' v + \gamma_\lambda^+(t) \partial_n v^+ - \gamma_\lambda^-(t) \partial_n v^-$.
Furthermore, using the bounds in \eqref{eq:boundg} and \eqref{Hp:condit-assum-flat}, it follows that the matrices $A^{\pm}_\lambda(t)$ and the functions $\hat{\gamma}_\lambda(t), \gamma_\lambda^{\pm}(t)$ satisfy the estimates
\begin{align*}
    &K^{-1}I \leq A_\lambda^{\pm}(t) \leq K I,  \quad
    |{\p_t} A_\lambda^{\pm}(t)| \leq K \lambda^{-1}, \\
    &|\hat{\gamma}_\lambda(t)| \leq K, \quad K^{-1} \leq \gamma_\lambda^{\pm}(t) \leq K, \quad
    |{\p_t} \hat{\gamma}_\lambda(t)| \leq \lambda^{-1}, \quad |{\p_t} \gamma_\lambda^\pm(t)| \leq \lambda^{-1}.
\end{align*}
 Thus, the limiting transmission problem \eqref{eq:LimitTwoStefan} matches the model \eqref{eq:main1}, with coefficients satisfying the structural assumptions \eqref{Assump1}. This formalizes the role of \eqref{eq:main1} as the linearized limiting profile of \eqref{eq:twophaseStef}.

\section{Preliminaries}\label{sec:preli}
\subsection{Notation}

First we define the notation we will use throughout the work: $(x,t)\in \R^{n+1}$, we call $x\in \mathbb{R}^n$ the space variable and $t \in \R$  the time variable. A point $x\in \mathbb{R}^n$ will sometimes be written as $x=(x',x_n)$, where $x'\in \R^{n-1}$ and $x_n \in \R$.
For a function $u(x,t)$, we denote by $\nabla u$ the gradient of $u$ with respect to $x$, $D^2 u$ the Hessian with respect to $x$, $\nabla' u$ the gradient with respect to $x'$, and $\p_t u$ the time derivative of $u$.  For a multi-index $m \in \mathbb{N}^n$, we denote the spatial derivative by $\partial_x^m u := \partial_{1}^{m_1} \dots \partial_{n}^{m_n} u$.
We call a \textit{parabolic cylinder} any set of the form 
 $$\mathcal{P}=U \times (t_1,t_2],$$
 where $U$ is a smooth bounded domain in $\R^n$ and $t_1 < t_2$. We denote by $\partial_p \mathcal{P}$ the \textit{parabolic boundary} of $\mathcal{P}$, i.e., $\partial_p {\mathcal{P}}=\left( U\times \{t_1\}  \right)  \cup   \left( \partial U \times [t_1,t_2] \right) $.
 
Given $(x,t), (y,s)\in \R^{n+1}$, we define the \textit{parabolic distance} as
$$
d_p((x,t), (y,s)) :=\big(|x-y|^2 + |t-s|\big)^{1/2}.
$$
Here and henceforth, given $r >0$, we set 
\begin{align*}
    Q_r &:= (-r,r)^n, & Q_r^{+} &:= Q_r \cap \{x_n > 0\}, \\
    Q_r^{-} &:= Q_r \cap \{x_n < 0\}, & Q_r(x_0) &:= x_0 + Q_r, \\
    \C_r &:= Q_r \times (-r,0], & \C_r^{\pm} &:= Q_r^{\pm} \times (-r,0], \\ 
    \C_r' &:= (Q_r \cap \{x_n=0 \}) \times (-r,0], & \C_r(x_0,t_0) &:= (x_0,t_0) + \C_r, \\
    \P_r &:= Q_r \times (-r^2,0], & \P_r^{\pm} &:= Q_r^{\pm} \times (-r^2,0], \\ 
    \P_r' &:= (Q_r \cap \{x_n=0 \}) \times (-r^2,0], & \P_r(x_0,t_0) &:= (x_0,t_0) + \P_r.
\end{align*}
Throughout the paper $\P_r$ and $\C_r$ respectively denote the parabolic and hyperbolic cylinders. 

We also define the Dirichlet boundary of $\C_r^{\pm}$ as 
\[
\p_D \C_r^{\pm}:= \p_p \C_r^{\pm} \cap \left( \{t=-r\} \cup_{i=1}^{n-1} \{ |x_i|=r \} \cup \{ x_n=\pm r \}\right).
\]

This represents the points in the parabolic boundary of $\C_r^{\pm}$ excluding the points of $\C_r'$.

Throughout the paper we say that a constant is universal if it depends only on $n$ and $K$. In what follows, $c, \, c_1, \, c_2,\, C, \, C_1,\, C_2, ...$ denote universal constants that may vary from line to line.

\subsection{Function spaces}
Let $C^{0}(\mathcal{P})$ be the set of continuous function in ${\mathcal{P}}$. Given $i,j \in \mathbb{N}$, $C_x^i(\mathcal{P})$ is the class of functions $u$ such that $\p_x^m  u \in C^0(\mathcal{P})$ for all $m$ multi-index, $|m| \leq i$ and $C_t^j(\mathcal{P})$ is the class of functions $u$ such that $\p_t^l  u \in C^0(\mathcal{P})$
for all integer $l \leq j$.  Given $k \in \mathbb{N}$, $C^k(\mathcal{P})$ is the class of functions $u$ such that $u \in C_x^i(\mathcal{P})\cap C_t^j(\mathcal{P}) $ such that $i + 2j \leq k$

For ${ \alpha} \in (0,1]$ and $u\in C^{0}(\mathcal{P})$, we define the parabolic H\"older semi-norm as
\begin{equation*}
    [u]_{C^{0,\alpha}(\mathcal{P})} := \sup_{\substack{(x,t)\neq (y,s)}} \frac{\nor{u(x,t)-u(y,s)}}{d_p((x,t),(y,s))^{{ \alpha}}},
\end{equation*}
and the parabolic H\"older norm as
\begin{equation*}
    \norm{u}_{C^{0,{ \alpha}}(\mathcal{P})} := \norm{u}_{L^{\infty}(\mathcal{P})} +[u]_{C^{0,{ \alpha}}({\mathcal{P}})}.
\end{equation*}
Further, we define the H\"older semi-norm in $t$ to be
\begin{equation*}
[u]_{C_t^{0,{ \alpha}}({\mathcal{P}})} := \sup_{ \substack{(x,t)\neq (x,s)}} \frac{|u(x,t)-u(x,s)|}{|t-s|^{ \alpha}}.
\end{equation*}
In particular it follows that if $u \in C^{0,{ \alpha}}(\mathcal{P})$ then $u$ is ${ \alpha}-$H\"older continuous in $x$ and ${ \alpha}/2-$H\"older continuous in $t$.
Introduce now, for $k \in \mathbb{N}$, the $C^{k,{ \alpha}}(\mathcal{P})$ space as the class of functions $u \in C(\mathcal{P})$ such that
\begin{align*}
      \|u\|_{C^{k,{ \alpha}}(\mathcal{P})} &:= \sum_{|m| + 2 j \leq k} \|\p_x^m \p_t^j u\|_{L^\infty(\mathcal{P})} + \sum_{|m| + 2j = k} [\p_x^m \p_t^j u]_{C^{0,{ \alpha}}(\mathcal{P})} \\
      &\qquad\qquad +  \sum_{|m| + 2 j = k-1} [\p_x^m \p_t^j u]_{C_t^{0,(1+{ \alpha})/2}(\mathcal{P})} <\infty. 
 \end{align*}
\subsection{Viscosity solutions}
Now  we define what we mean by viscosity solution to \eqref{eq:main1}.
\begin{Definition}  \label{def:visc}
We say that $u \in USC(\C_1)$ is a viscosity subsolution to \eqref{eq:main1} if the following holds:
\begin{itemize}
    \item[(i)] If $(x_0,t_0) \in \C_1^{\pm}$ and $\varphi \in C^{2}(\C_\delta(x_0,t_0) )$ touches $u$ from above at $(x_0,t_0)$ then 
    \[ 
   \lambda \p_t \varphi(x_0,t_0) - \tr(A^\pm(t_0) D^2 \varphi(x_0,t_0)) \leq 0 
    \]
    \item[(ii)] If $(x_0,t_0) \in \C_1'$ and  $\varphi \in C_x^{1}(\overline{ \C_\delta^+(x_0,t_0)}) \cap  C_t^{1}(\overline{ \C_\delta^+(x_0,t_0)})$ and $\varphi \in C_x^{1}(\overline{ \C_\delta^-(x_0,t_0)}) \cap  C_t^{1}(\overline{ \C_\delta^-(x_0,t_0)})$  touches $u$ from above  then 
    \[
\p_t\varphi(x_0,t_0) \leq  -\hat{\gamma}(t_0) \cdot  \nabla' \varphi(x_0,t_0)  + \gamma^+( t_0) \p_n \varphi^+(x_0,t_0) - \gamma^-(t_0) \p_n \varphi^-(x_0,t_0)  .
    \]
\end{itemize}
The notions of viscosity supersolution and solution are defined as usual. Moreover if the inequalities in the definition are strict, we refer to them as strict subsolution or  supersolution respectively.
\end{Definition} 

Notice that in Definition \ref{def:visc} (ii), test functions are required to be of class $C_x^1$ up to the interface $\mathcal{C}_1'$, which is the natural regularity for the transmission condition to be classically defined. However, by adding a quadratic penalization term and a standard approximation argument, we can restrict the test functions in (ii) to be $C_x^2$ up to the boundary yielding an equivalent definition. Throughout the paper, we will freely assume test functions at the interface to be $C_x^2$ from both sides.
    
In particular the operator $F^\pm(M,t):=\tr(A^\pm(t)M)$ belongs to a more general class of uniformly elliptic operators for $t \in (-1,0]$.
Indeed, we define the operator $F(\cdot,t):\mathcal{S}(n)\to\mathbb{R}$ to be uniformly $K$-elliptic for every $t \in (-1,0]$, i.e, there exists $K>1$ universal such that 
\[
	K^{-1} |N|\leq F(M+N,t)-F(M,t) \leq K|N|,
\]
for every $M,N\in\mathcal{S}(n)$, with $N\geq 0$.
Recall the definition of extremal Pucci operators with ellipticity constants $K^{-1}$ and $K$  
\begin{equation*}\label{Pucci} \mathcal M_K^+ (N) = \sup_{K^{-1} I \le A \le K I} \quad \tr (A N), \quad \quad  \quad \mathcal M_K^- (N) = \inf_{K^{-1} I \le A \le K I} \quad \tr ( A N).\end{equation*} 
Let $f^\pm \in L^\infty(\C_{1}^\pm)\cap C(\C_{1}^\pm)$, 
the class $\underline{S}(f^\pm)$ denotes the functions $u \in USC(\C_1)$ such that $\p_t u - \mathcal M_K^+ (D^2 u) \leq f^\pm $ in $\C_1^{\pm}$ and the definitions of the classes $\overline{S}(f^\pm), S (f^\pm)$ and $S^{*}(f^\pm)$ follow as in \cite{CC}.

 We also recall the Weak Harnack inequality for fully nonlinear parabolic equations  \cite[Corollary 4.14]{wang1992regularityI}.
    \begin{Lemma}[Parabolic Weak Harnack inequality] \label{Lm:WeakHarnackIneq}
        Let $f \in L^{\infty}(\P_1) \cap C(\P_1)$ and let $u \in \overline{S}(f)\mbox{ in } \P_1$ be a nonnegative function. 
        Then for every $\Omega \Subset Q_1$ smooth it holds
        \[
        \left(\fint_{\Omega \times(-3/4, -1/2)} u^{p_0}\right)^\frac{1}{p_0}\leq C_H\left(\inf_{\Omega\times(-1/4, 0)} u + \|f\|_{L^\infty(\P_1)}\right)
        \]
        where $p_0 \in (0,1),\, C_H>0$ depend only on $n, \, K$ and $\Omega$.
    \end{Lemma}

It will be useful to introduce the following class of solutions to extremal transmission problems which includes \eqref{eq:main1}. More precisely we denote by $\underline{S}_{\lambda,T}(f^{\pm})$ the functions $u \in USC(\C_1)$ that solve
\begin{equation} \label{eq:subsol}
    \begin{cases}
        \lambda \p_t u - \mathcal M_K^+ (D^2 u) \leq f^\pm  & \text{ in } \C_1^{\pm}, \\ 
          \p_t u \leq K |\nabla' u| + K\big( (\p_n u^+)_+ + (\p_n u^-)_- \big) - K^{-1}\big( (\p_n u^+)_- + (\p_n u^-)_+ \big) & \text{ on }\C_{1}'.
    \end{cases}
\end{equation}
In a similar way, we introduce the class $\overline{S}_{\lambda,T}(f^{\pm})$ the functions $u \in LSC(\C_1)$ that solve
\begin{equation} \label{eq:supersol}
    \begin{cases}
        \lambda \p_t u - \mathcal M_K^- (D^2 u) \geq f^\pm  & \text{ in } \C_1^{\pm}, \\ 
          \p_t u \geq  -K |\nabla' u| + K^{-1}\big( (\p_n u^+)_+ + (\p_n u^-)_- \big) - K\big( (\p_n u^+)_- + (\p_n u^-)_+ \big)& \text{ on }\C_{1}'.
    \end{cases}
\end{equation}
Analogously as above, we define
\begin{align*}
    &S_{\lambda,T}(f^{\pm}):=  \underline{S}_{\lambda,T}(f^{\pm})  \cap  \overline{S}_{\lambda,T}(f^{\pm})\\
    &S^*_{\lambda,T}(f^{\pm}):=\underline{S}_{\lambda,T}(-\|f^{\pm}\|_{L^{\infty}})  \cap \overline{S}_{\lambda,T}(\|f^{\pm}\|_{L^{\infty}}).
\end{align*}

Up to dividing the solution by a suitable positive constant 
we can assume without loss of generality that $\|u\|_{L^\infty(\C_1)} \leq 1$ and $\|f^\pm\|_{L^\infty(\C_1)} \leq \varepsilon_0$ for some small constant $\varepsilon_0 > 0$.

\section{Proof of Theorem \ref{Lm:Harnack}: Harnack Inequality} \label{sec:harnack}
In this section we derive the Harnack inequality of Theorem \ref{Lm:Harnack} for elements of the class $S_{\lambda,T}(f^{\pm})$. In fact, Theorem \ref{Lm:Harnack} follows directly from the following result. 

\begin{Lemma}\label{Lm:Harnack1}
    Let $u\in S^*_{\lambda,T}(f^{\pm})$  in $\C_1$  with $\|f^\pm\|_{L^\infty(\C_1)} \leq \varepsilon_0$ and such that $0\leq u  \leq 1$ in $\C_1$.  Then 
\[
\osc_{\C_{1/2}} u \leq 1-c,
\]
 where $\eps_0\leq c_1 \lambda$ and $c,c_1>0$ are small universal constants.
\end{Lemma}

By applying Lemma \ref{Lm:Harnack1} to 
\[
v(x,t) := \frac{u(x,t) - \inf_{\C_1} u}{M}.
\]
with $M= \osc_{\C_1} u + \frac{1}{c_1 \lambda}\|f^\pm\|_{L^\infty(\C_1)}$ we obtain Theorem \ref{Lm:Harnack}. 

We devote the rest of the section to the proof of Lemma \ref{Lm:Harnack1}.
 
We first introduce the following barrier in the space variables.
Let $\sigma>0$ be a dimensional constant such that $|Q_1^+\setminus Q_{1-\sigma}^+| \leq \frac{1}{16}|Q_1^+|$. Let $\Omega^+$ be a smooth domain such that $ \overline{Q_{1-2\sigma}^+} \subset  \overline{\Omega^+} \subset \overline{Q_{1-\sigma}^+} $ and define $T:= \{ x_n=0 \} \cap Q_{1-2\sigma} \subset \p \Omega^+$.
We consider $\eta$ to be a cut-off function in the $x'$-variables such that $\supp \eta \subset Q_{5/8}'$ and $ \eta \equiv 1 \subset Q_{1/2}'$. Define $\phi^+$ to satisfy the following Dirichlet problem in $\Omega^+$
\begin{equation} \label{eq:barrierphi}
\begin{cases}
           \mathcal M_K^- (D^2 \phi^+) = 0 & \text{ in }\Omega^+,\\
        \phi^+ =0  & \text{ in }\p \Omega^+ \setminus T, \\
        \phi^{+} = \eta  & \text{ in } T.
\end{cases}
\end{equation}
From the maximum principle it follows that $0 \leq \phi^+ \leq 1$ and $\phi^+ \geq \bar c$ on $Q_{1/2}^+$. By Hopf lemma we also have $\p_n \phi^+ >0$ on $\{ x_n = 0\} \cap \{ \phi^+ =0\}$. Furthermore, since $\Mm$ is concave, in view of \cite[Chapter $9$]{CC}, we have that $\phi \in C^{2,{ \alpha}}(\Omega^+)$ with universal estimates. 
On the set $\Omega^{-}$, which we define to be the reflection of $\Omega^+$ with respect to $\{x_n=0\}$, we set $\phi^{-}(x',x_n)=\phi^{+}(x',-x_n)$ and call
\[
\phi(x) = \phi^+(x) \chi_{\overline{\Omega^+}}(x) + \phi^-(x) \chi_{\Omega^-}(x).
\]
It is straightforward to notice that $D^2 \phi^+$ and $D^2 \phi^-$ have the same eigenvalues since $D^2\phi^- = OD^2\phi^+O^T $, where $O = \text{diag}(1, 1, \dots, 1, -1)$ ($O$ is the orthogonal matrix representing the reflection across the hyperplane defined by $\{x_n=0\}$). Thus $ \mathcal M_K^- (D^2 \phi^-)=0$ in $\Omega^-$.

The following Lemma \ref{Lm:Barrier_improvement} is the key propagation result to obtain the Harnack inequality. The first part  establishes that a positive multiple of the elliptic barrier \(\phi\) can be transported forward in time, up to a small error
\(\varepsilon_0/\lambda\). 
The second part shows that if, during a time
interval of length comparable to \(\lambda\), the solution is positive
on a set of large measure in one side of the interface, then the coefficient of the
barrier improves by a universal multiple of \(\lambda\). The proof first
uses the weak Harnack inequality away from the interface and then
transfers the gain across the interface by means of the transmission
condition.

\begin{Lemma} \label{Lm:Barrier_improvement}
    Let $u \in \overline{S}_{\lambda,T}(f^{\pm})$ in $\C_1$  and $u\in C(\C_1)$ with  $\|f^\pm\|_{L^\infty(\C_1)} \leq \varepsilon_0$ and such that $u \geq 0$ in $\C_1$.
If for some $t_0 \in (-1,0]$ and $s_0 \geq 0$ it holds that 
\begin{equation} \label{eq:initial_time}
    u(x,t_0) \geq s_0 \phi(x) \quad \text{in }Q_1,
\end{equation} 
then 
\begin{equation} \label{eq:boundspacetime}
    u(x,t) \geq s(t) \phi(x)  {- \frac{\eps_0}{\lambda}(t-t_0)}  \quad \text{in }Q_1 \times [t_0,0],
\end{equation}
    with 
    \[
    \p_t s(t)=-C_0s(t), \quad s(t_0)=s_0, \quad C_0\text{ large universal.}
    \]
Moreover, if $s_0 \leq c_0$ for a  small universal constant $c_0$ and for $t_0<-\lambda$ it holds
\begin{equation} \label{eq:meas_est_harn}
    \left| \left\{ u \geq \frac{1}{4} \right\} \cap \left( Q_1^+ \times \left[t_0, t_0+\frac{\lambda}{8} \right]\right) \right| \geq \frac{1}{2} \left|Q_1^+ \times \left[t_0, t_0+\frac{\lambda}{8} \right] \right|,
\end{equation}
    then
    \begin{equation} \label{eq:improv2}
        u(x,t_0 + \lambda) \geq (s_0+c_0 \lambda)\phi(x) { - \eps_0}  \quad \text{ in }Q_{1}.
    \end{equation}
\end{Lemma}

\begin{proof} 
    Let $\Omega= \overline{\Omega^+} \cup \Omega^-$, define 
    \[
    w(x,t)=s(t)\phi(x)-\frac{\eps_0}{\lambda}(t-t_0) \quad \text{ in } \Omega \times [t_0,0] \] 
    and extend it by $-\frac{\eps_0}{\lambda}(t-t_0)$ in the rest of $Q_1 \times [t_0,0]$. For the first part of the claim we notice that if $s_0=0$ the statement holds trivially since $u \geq 0$.  For $s_0>0$, we want to compare $u$ and $w$ showing that the latter is a subsolution to \eqref{eq:supersol} in $\Omega \times [t_0,0]$. Note that  such comparison holds on $\p_D \left( \Omega \times [t_0,0]\right)$ since by construction $\phi$ vanishes on $ \p \Omega \setminus T$ and on $\{t=t_0 \}$ it  holds \eqref{eq:initial_time}.
    The interior equations in $\Omega^\pm \times (t_0,0]$ are satisfied since 
    \[
    \lambda \p_t w - \mathcal M_K^- (D^2 w) = \lambda s'(t) \phi(x) - \eps_0 - s(t)\mathcal M_K^- (D^2 \phi) \leq -\eps_0,
    \]
    using that $\p_t s \leq 0$ and $s\geq 0$.  On the transmission interface $\{ x_n=0\}$, we need to show that 
    \[ 
    \p_t w \leq -K |\nabla' w| + K^{-1}\left( (\p_n w^+)_+ + (\p_n w^-)_- \right) - K\left( (\p_n w^+)_- + (\p_n w^-)_+ \right). 
    \]
    Since on $\{ x_n=0\}$ we have $ \p_n w^-(x',0,t)=-s(t) \p_n \phi^+(x',0,t)$, the latter condition follows if we prove
    \[
    -K|\nabla' \phi| +2K^{-1} (\p_n \phi^+)_+ - 2K ( \p_n \phi^+)_- + C_0 \phi \geq 0.
    \]
    By Hopf lemma, we have that both $\p_n \phi^+ > 0$ and $|\nabla'\phi|=0$ hold on $\{ \phi = 0 \} \cap \{x_n = 0 \}$, which implies that
    \[
    -K|\nabla' \phi| +2K^{-1} (\p_n \phi^+)_+ - 2K (\p_n \phi^+)_- + C_0 \phi = 2K^{-1} \p_n \phi^+ >0 \quad \text{ in }\{ \phi = 0 \} \cap \{x_n = 0 \}. 
    \]
    The same inequality holds in a neighborhood of this set by continuity. In the rest of $\{x_n= 0\} \cap \Omega$ we can choose $C_0$ large enough since $\phi >0$ giving \eqref{eq:boundspacetime}.
    
    For the second part of the claim we divide the proof in two steps. In the first step we prove the bound on the upper part of the cylinder. In the second part we carry the improvement to the other side at a small universal distance to the interface. Then the estimate is spread in the whole lower part as in the first step.

\underline{\textit{Step 1.}} We first prove that \eqref{eq:improv2} holds in $Q_1^+$ at the time $t=t_0+\lambda$ following the argument of \cite[Lemma 6.2]{DFS23}.

Set \(t_i:=t_0+i\lambda/8\) for \(i=1,\dots,8\) and let $D^+$ be given by  
\[ 
D^+:= \{ x \in \Omega^+ : \operatorname{dist}(x,\p \Omega^+) > c_1\} 
\]
 where the value of $c_1$ is chosen universally small for the existence of the function $\psi$ in $\Omega^+ \setminus D^+$ such that 
 \[
    \mathcal M_K^- (D^2 \psi) \geq 16 \quad \text{ in }\Omega^+ \setminus \overline{D^+}, 
 \]
and
\[
\psi= 0, \ |\nabla \psi| \geq 1 \; \text{ on } \p \Omega^+, \quad \psi\leq 1 \; \text{ on } \p D^+.
\]
 By making $c_1$ smaller we ensure that $|\Omega^+\setminus D^+| \leq \frac{1}{16}|Q_1^+|$. Thus from \eqref{eq:meas_est_harn} we deduce that $\left|\{ u \geq \frac{1}{4}\}\cap ( D^+  \times \left[t_0, t_0+\frac{\lambda}{8} \right] ) \right| \geq \frac{1}{4} \left| Q_1^+ \times \left[t_0, t_0+\frac{\lambda}{8} \right] \right|$. 
If we define $v(x,t):=u(x,t_0+\lambda t)$, then we understand that $v \in \overline{S}(f^+)$ in $Q_1 \times (0,1]$. Hence, by Weak Harnack inequality of Lemma \ref{Lm:WeakHarnackIneq}, we obtain
\begin{equation} \label{eq:est_in_D+}
\begin{aligned}
    \inf_{D^+ \times (t_2,t_8)} u &=   \inf_{D^+ \times (\frac{1}{4},1)} v \geq C_H^{-1} \left(\frac{1}{|D^+\times(0,1/8)|} \int_{\{ v \geq \frac{1}{4}\} \cap  \left(D^+\times(0,1/8) \right)} v^{p_0}\right)^\frac{1}{p_0} -\eps_0\\ &=C_H^{-1} \left(\frac{1}{|D^+\times(t_0,t_1)|}\int_{\{ u \geq \frac{1}{4}\} \cap  \left(D^+ \times(t_0,t_1) \right)} u^{p_0}\right)^\frac{1}{p_0} -\eps_0 \geq 2c_2
    \end{aligned}
    \end{equation}
for $c_2 \leq 2^{-1/p_0-4}C_H^{-1}$ and $\eps_0\leq c_2$. In order to produce a bound in the remaining narrow region $(\Omega^+ \setminus D^+) \times [t_2,t_8]$, we compare with explicit barriers. In  $(\Omega^+ \setminus D^+) \times [t_2,t_3]$, we define
\[
q(x,t):= s(t_3) \phi(x) + c_2 \left( \psi(x) + \frac{t-t_3}{t_3-t_2}\right) {-\frac{\eps_0}{\lambda}(t-t_0)}.  
\]
On the boundary of the domain the inequality $q \leq u$ holds. Indeed, on $\p D^+$, since $s$ is decreasing, we have 
\[
q(x,t) \leq s(t_3) \phi + c_2 {-\frac{\eps_0}{\lambda}(t-t_0)}\leq s_0+c_2 \leq 2c_2 \leq u(x,t)
\]
for $s_0 \leq c_0$ sufficiently small and using \eqref{eq:est_in_D+}. Moreover for $x \in \p \Omega^+$ we know $\psi = 0$ and thus
\[
q(x,t) \leq s(t_3) \phi(x)  {-\frac{\eps_0}{\lambda}(t-t_0)} \leq s(t) \phi(x) {-\frac{\eps_0}{\lambda}(t-t_0)}\leq u(x,t),
\]
from the first part of the lemma. The same bound holds at $t=t_2$ using that $\psi \leq 1$.
We now check that $q$ is a subsolution for the interior equation. In $(\Omega^+ \setminus D^+) \times (t_2,t_3]$ we have
\begin{align*}
    \lambda \p_t q  &= \frac{\lambda c_2}{t_3-t_2} {-\eps_0} \leq 16c_2 -\eps_0 \leq c_2\mathcal{M}_K^-(D^2 \psi) -\eps_0 \\
    &\leq c_2  \mathcal{M}_K^-(D^2 \psi) 
    + \mathcal{M}_K^-((s(t_3))D^2 \phi) -\eps_0 \leq \mathcal{M}_K^-(D^2 q)-\eps_0.
\end{align*}
Hence, comparing $u$ and $q$ at $t=t_3$ we deduce that
\begin{equation} \label{eq:compare_u_q}
    u(x,t_3) \geq s(t_3) \phi(x) + c_2 \psi(x) {-\frac{3}{8}\eps_0}\quad \text{in } \Omega^+ \setminus D^+.
\end{equation}
We proceed comparing $u$ in $\left( \Omega \setminus D^+\right) \times[t_3,t_8]$ with the following classical strict subsolution to \eqref{eq:supersol}
\[
z(x,t):=
\begin{cases}
    \left(s(t_3) + c_3(t-t_3)\right) \phi(x) + c_2 \psi(x) {-\frac{\eps_0}{\lambda}(t-t_0)}, &\text{for } (x,t)\in \left( \Omega^+ \setminus D^+\right) \times[t_3,t_8]\\
    \left(s(t_3) + c_3(t-t_3)\right) \phi(x) + \tilde{c_2} \chi(x) {-\frac{\eps_0}{\lambda}(t-t_0)}, &\text{for } (x,t)\in  \overline{\Omega^-} \times[t_3,t_8]
\end{cases}
\]
with $\tilde c_2, c_3$ sufficiently small. Here, $\chi$ is the unique classical solution to
\[
    \mathcal M_K^-(D^2 \chi)=2 \quad \text{ in } \Omega^- \quad \text{ and } \quad \chi=0 \text{ on } \p\Omega^-.
\]
In particular, it satisfies $\chi\leq 0$, and on $\{x_n=0\}$ we have  $\nabla' \chi=0$, $0 <\p_n \chi\leq L$ where $L>0$ is universal.

From \eqref{eq:compare_u_q} and $\chi \leq 0$ we know that $u \geq z$ at time $t=t_3$. The same holds on $\p \Omega$ since both $\phi,\, \psi$  and $\chi$ vanish and $u\geq0$. Moreover, on $\p D^+$ we have
\[
z \leq s_0 + c_3\lambda+c_2 -\frac{\eps_0}{\lambda}(t-t_0) \leq 2c_2 \leq u,
\]
for $c_0$ and $c_3$ small enough by \eqref{eq:est_in_D+}.
Regarding the interior equations we have, on $\Omega^+\setminus \overline {D^+}$
\[
\lambda \p_t z = \lambda c_3 \phi -\eps_0 \leq c_3 -\eps_0 \leq c_2 \mathcal{M}_K^-(D^2 \psi) -\eps_0 \leq \mathcal{M}_K^-(D^2 z)-\eps_0.
\]
A similar computation verifies the condition in $\Omega^-$ for $c_3 \leq \tilde c_2$.

It remains to check the condition on $\{ x_n=0\}$. 
Using that $\p_n \psi \geq 1$ on $\p \Omega^+ \cap \{x_n=0 \}$ and $\p_n\chi\leq L$ we get
\begin{equation}
    \p_t z = c_3 \phi   -\frac{\eps_0}{\lambda} \leq \frac{c_2}{8} K^{-1} \p_n \psi.
\end{equation}
Moreover, since $\p_n \phi^+ >-C$, for $c_0$ and $c_3$ small, we obtain
\begin{align*}
    \p_n z^+&\geq - (s_0 + c_3\lambda)C + c_2 \p_n\psi \geq \frac{c_2}{2} \p_n\psi > 0 \\
    \p_n z^- &\leq (s_0 + c_3\lambda)C +\tilde c_2 L \leq \frac{c_2}{8}K^{-2} \p_n\psi,
\end{align*}
by choosing $\tilde c_2$ sufficiently small depending also on $c_2$.
Since $\nabla'\psi= 0,\nabla' \chi=0$ on $\{x_n=0 \}$ we also have
\[
K|\nabla' z| \leq (s_0 + c_3\lambda)K |\nabla' \phi| \leq \frac{c_2}{8} K^{-1} \p_n \psi .
 \]
Combining these estimates we reach 
\begin{align*}
   -K|\nabla 'z| &+ K^{-1}  \left( (\p_n z^+)_+ + (\p_n z^-)_- \right) - K \left( (\p_n z^+)_-  +(\p_n z^-)_+ \right) \\
   &\geq 
-\frac{c_2}{8} K^{-1} \p_n \psi  +\frac{c_2}{2}K^{-1} \p_n \psi  -\frac{c_2}{8} K^{-1} \p_n \psi \geq \frac{c_2}{4}K^{-1} \p_n \psi \geq \p_t z \quad \text{ on }\{x_n=0\}.
\end{align*}
Finally by comparing $u$ and $z$ and with \eqref{eq:est_in_D+}, we obtain for $t \in (t_3, t_8)$
\begin{equation} \label{eq:1111}
    u(x,t) \geq \begin{cases}
        \left(s(t_3) + c_3(t- t_3) \right) \phi(x) + c_2 \psi(x)  {-\frac{\eps_0}{\lambda}(t-t_0)} \quad &\text{ in } \Omega^+ \setminus D^+, \\
        2c_2 \quad &\text{ in } D^+, \\
    \end{cases}
\end{equation}
which, choosing $c_0$ small, gives
\begin{equation}
    u(x,t_0+\lambda) \geq \left( s_0 + c_0 \lambda\right) \phi(x) -\eps_0 \quad \text{ in  }Q_1^+. 
\end{equation}

\underline{\textit{Step 2.}}
Inspecting the proof of the first step, denoting by $\bar t= t_3 + \frac{\lambda}{16}$, from \eqref{eq:1111} we deduce that on the interface $\{x_n = 0\}$ it holds 
\begin{equation} \label{eq:1step_2}
    u(x,t) \geq \left( s(t) + c_3 \frac{\lambda}{16} \right) \phi(x) -\frac{\eps_0}{\lambda}(t-t_0)\quad \text{ for all } t \in \left[ \bar t, t_8\right],
\end{equation}
for a universal constant $c_3$ independent of $c_0$.
To transfer the estimate to the other side of the interface, we introduce the function
\[
v(x,t) := u(x,t) - s(t)\phi(x) + \frac{\eps_0}{\lambda}(t-t_0) \geq 0 \quad \text{ in } Q_1.
\]
By the first part of the lemma and since $s'(t) \leq 0$, $v$ is a nonnegative supersolution to the homogeneous equation $\lambda \p_t v - \mathcal M_K^- (D^2 v) \geq 0$ in $\Omega^- \times (t_0, 0]$. Moreover, \eqref{eq:1step_2} translates to the boundary estimate
\begin{equation} \label{eq:bound_v_interface}
    v(x,t) \geq c_3 \frac{\lambda}{16} \phi(x) \quad \text{ on } \p \Omega^- \cap \{ x_n = 0\} \text{ for } t \in [\bar t, t_8].
\end{equation}

We define the set $U:= \Omega^- \cap \{ x_n \geq -c_4\}$ for $c_4$ small universal such that there exists a nonnegative $C^2$ function $\xi$ defined in $U$ satisfying 
\begin{equation}
    \begin{cases}
        \mathcal M_K^- (D^2 \xi) \geq 16 \quad &\text{ in }U, \\
        \xi = \phi \quad &\text{ on } \p U \cap \{x_n=0\},  \\
        \xi= 0 \quad &\text{ on } \p U \setminus \{x_n=0\}.
    \end{cases}
\end{equation}
By maximum principle $0 \leq \xi \leq \phi \leq 1$ in $U$. We compare $v$ in $U \times [\bar t,t_4]$ with the barrier
\[
\omega(x,t):= c_5 \lambda \xi(x) + c_5 \lambda \frac{t-t_4}{t_4 - \bar t},
\]
where $c_5$ is a small universal constant. On $\{ x_n=0\}$, imposing $c_5 \leq c_3/16$, we have $\omega(x,t) \leq c_5 \lambda \phi(x) \leq v(x,t)$ by \eqref{eq:bound_v_interface}.
On $\left(\p U \setminus \{x_n=0\}\right) \times [\bar t,t_4]$ , the comparison follows since $\omega \leq 0 \leq v$. At the initial time $t=\bar t$, using that $\xi \leq 1$ in $U$, we get $\omega(x,\bar t) = c_5 \lambda (\xi(x)-1) \leq 0 \leq v(x,\bar t)$.
It remains to check that $\omega$ is a subsolution, which is readily verified since 
$$\lambda \p_t \omega = 16 c_5 \lambda \leq c_5 \lambda \mathcal M_K^- (D^2 \xi) = \mathcal M_K^- (D^2 \omega).
$$
Thus $v \geq \omega$, yielding at $t=t_4$
\[
v(x,t_4) \geq  c_5 \lambda \xi(x) \quad \text{ in } U.
\]

From the continuity of $\xi$ there exists a small universal constant $c_6 < c_4$, such that 
$$
\xi \geq \phi/2 \text{ in } U_1:=\left\{ x \in U \, : \, \operatorname{dist}\!\left(x, \p U \setminus \{ x_n=0\}\right) > c_4-c_6 \right\} \cap Q_{1/2}^-.
$$
Defining 
\[
D^-:= \{ x \in \Omega^- : \operatorname{dist}(x,\p \Omega^-) > c_6/4\},
\]
we apply the Weak Harnack inequality of Lemma \ref{Lm:WeakHarnackIneq} to the rescaling $\tilde v(x,t)=v(x,t_3+\lambda t)$. Since $\tilde v \in \overline{S}(0)$ in $D^- \times (0,5/8]$, in terms of $v$ we obtain
\begin{align} \label{eq:est_in_D^-}
\inf_{D^- \times (t_5,t_8)} v &\geq C_H^{-1}\left( \frac{1}{|D^-\times (\bar t, t_4)|}\int_{ (D^- \cap U_1)\times(\bar t, t_4)} v^{p_0}\right)^\frac{1}{p_0} \geq  c\frac{C_H^{-1}}{2} c_5 \lambda \inf_{Q_{1/2}^-}\phi \geq 2 c_7 \lambda,
\end{align}
where $c_7>0$ is a small universal constant independent of $c_0$. 

To propagate this improvement to the rest of $Q_1^-$ up to the final time $t_8 = t_0 + \lambda$, we repeat the explicit barrier argument of Step $1$ using the bound \eqref{eq:1step_2} as a Dirichlet datum on the interface $ \p \Omega^- \cap \{ x_n=0\}$.

Let $\psi_0$ be a smooth function in $\Omega^- \setminus D^-$ satisfying $\mathcal M_K^- (D^2 \psi_0) \geq 16$, with $\psi_0 = 0, \ |\nabla \psi_0| \geq 1 \; \text{ on } \p \Omega^-$ and $\psi_0 \leq 1$ on $\p D^-$. 
In $(\Omega^- \setminus D^-) \times [t_6,t_7]$, we compare $v$ with
\[
q_0(x,t):= c_7\lambda \left( \psi_0(x) + \frac{8(t-t_7)}{\lambda}\right).
\]
On $\p D^-$, we have $q_0 \leq c_7\lambda \leq 2 c_7\lambda \leq v$.  On the remaining parabolic boundary, $q_0 \leq 0 \leq v$. Since $\lambda \p_t q_0 = 8 c_7 \lambda \leq c_7 \lambda \mathcal{M}_K^-(D^2 \psi_0) \leq \mathcal{M}_K^-(D^2 q_0)$, the comparison principle yields $v \geq q_0$, which at $t=t_7$ gives
\begin{equation} \label{eq:compare_v_q-}
    v(x,t_7) \geq c_7\lambda\psi_0(x) \quad \text{in } \Omega^- \setminus D^-.
\end{equation}

Finally, we compare $v$ in $(\Omega^- \setminus D^-) \times [t_7,t_8]$ with
\[
z_0(x,t) := c_8(t-t_7) \phi(x) + c_7 \lambda \psi_0(x),
\]
for a universal constant $c_8 > 0$. At $t=t_7$, $z_0 \leq v$ by \eqref{eq:compare_v_q-}. On $\p \Omega^- \setminus \{ x_n=0\}$, $z_0 \leq 0 \leq v$. On $\p D^-$, choosing $c_8$ small depending on $c_7$, we get $z_0 \leq c_8 \frac{\lambda}{8} + c_7 \lambda \leq 2c_7\lambda \leq v$. On the interface $\{ x_n=0\}$, using \eqref{eq:bound_v_interface}, we require $c_8(t-t_7)\phi \leq c_3 \frac{\lambda}{16}\phi$, which holds by setting $c_8 \leq c_3/2$.
The interior equation requires $\lambda \p_t z_0 = \lambda c_8 \phi \le 16 c_7 \lambda \le \mathcal M_K^- (D^2 z_0)$, satisfied if $c_8$ is small depending on $c_7$. 
Thus $v \ge z_0$. Evaluating at $t_8 = t_0+\lambda$, we obtain $v(x,t_0+\lambda) \geq c_8\frac{\lambda}{8}\phi(x)$ in $\Omega^- \setminus D^-$. Substituting back $u$, we reach
\begin{equation} \label{eq:est_narrow_region}
u(x,t_0+\lambda) \geq \left( s(t_8) + c_8\frac{\lambda}{8}\right)\phi(x) - \eps_0 \quad \text{ in }\Omega^- \setminus D^-.
\end{equation}
On the other hand, in $D^-$, the Weak Harnack estimate \eqref{eq:est_in_D^-}, using that $\phi(x) \leq 1$, gives  
\begin{equation} \label{eq:est_core_region}
u(x,t_0+\lambda) \geq \left( s(t_8) + 2c_7\lambda\right)\phi(x) - \eps_0 \quad \text{ in } D^-.
\end{equation}
Combining \eqref{eq:est_narrow_region} and \eqref{eq:est_core_region}, and using the ODE bound $s(t_8) \geq s_0 - C_0 s_0 \lambda$, the coefficient of $\phi(x)$ in the whole $Q_1^-$ is bounded below by 
\[
s_0 - C_0 s_0 \lambda + \min\left\{ \frac{c_8}{8}, 2c_7 \right\} \lambda.
\]
Because $s_0 \leq c_0$, we can choose $c_0$ universally small such that $c_0(C_0 + 1) \leq \min\left\{ \frac{c_8}{8}, 2c_7 \right\}$. This guarantees 
\[
u(x,t_0+\lambda) \geq (s_0 + c_0 \lambda)\phi(x) - \eps_0 \quad \text{ in } Q_1^-,
\]
which, combined with the estimate in $Q_1^+$ obtained in Step 1, yields the desired bound \eqref{eq:improv2} in the whole $Q_1$.

     \end{proof}

Iterating Lemma \ref{Lm:Barrier_improvement} we can now provide the proof of Lemma \ref{Lm:Harnack1}.

\begin{customproof}{Lemma \ref{Lm:Harnack1}}
    
    Define the times $t_k:=-1+k\lambda$, for $k=0, \dots, N-1$ where \(N\) be the largest even integer satisfying \(N\lambda\le1/2\) and thus $t_k\in[-1,-1/2)$. We can assume, without loss of generality, that for at least half of the values of $t_k$ we have 
    \begin{equation} \label{eq:meas_est_harn_tk}
    \left| \left\{ u \geq \frac{1}{4} \right\} \cap \left( Q_1^+ \times \left[t_k, t_k+\frac{\lambda}{8} \right]\right) \right| \geq \frac{1}{2} \left|Q_1^+ \times \left[t_k, t_k+\frac{\lambda}{8} \right] \right|.
    \end{equation}
    Indeed, if this does not hold, then we can consider $1-u$ which, by assumption, is still nonnegative and belongs to the same class. Then arguing in the same way for $1-u$ one gets the desired result. With \eqref{eq:meas_est_harn_tk} in force, we want to show that we can apply iteratively Lemma \ref{Lm:Barrier_improvement} to separate $u$ from zero in $\C_{1/2}$. 
    
    Define the shifted sequence of functions
    \[
        u_k(x,t) := u(x,t) + k\eps_0.
    \]
    Clearly $u_k \in \overline{S}_{\lambda,T}(f^\pm)$ for every $k$ and since $u_k \geq u$, the measure condition \eqref{eq:meas_est_harn_tk} trivially holds for $u_k$ whenever it holds for $u$.

    We want to use Lemma \ref{Lm:Barrier_improvement} at each time $t_k$, with $s_0=0$ and $s_k$ acting as the lower bound coefficient for $u_k$. We begin by explaining the argument. Suppose we have proved $u_k(x,t_k) \geq s_k\phi(x)$ in $Q_1$. If \eqref{eq:meas_est_harn_tk} holds (good case) and $s_k \leq c_0$, then we use the second part of Lemma \ref{Lm:Barrier_improvement} applied to $u_k$ to obtain the improvement
    \[
    u_k(x,t_{k+1}) \geq (s_k+c_0\lambda)\phi(x) - \eps_0 \quad \textit{in }Q_1.
    \]
    Substituting $u_k = u + k\eps_0$, this rewrites as
    \[
    u(x,t_{k+1}) + (k+1)\eps_0 \geq (s_k+c_0\lambda)\phi(x)  ,
    \]
    that is $u_{k+1}(x, t_{k+1}) \geq s_{k+1}\phi(x)$  with $s_{k+1} = s_k + c_0\lambda$.
    On the other hand, if \eqref{eq:meas_est_harn_tk} does not hold (bad case), then, from the first part of Lemma \ref{Lm:Barrier_improvement}, we have the worse bound
    \[
    u_k(x,t_{k+1}) \geq s_k(1-C_0\lambda)\phi(x) - \eps_0. 
    \]
    Similarly, this yields $u_{k+1}(x, t_{k+1}) \geq s_{k+1}\phi(x)$ with $s_{k+1} = s_k(1-C_0\lambda)$.

    Since \eqref{eq:meas_est_harn_tk} holds more than half of the time, the final control we get on $s_k$ will depend on the order of the $t_k$'s where \eqref{eq:meas_est_harn_tk} is satisfied. Clearly, the worst case scenario is when it holds for the first half of times $t_k$. We only present the estimate in this case.

    We start with $s_0=0$. Then, we apply the good case $k_1$ times, where $k_1=N/2$, to get
    \[
    s_{k_1}= k_1c_0\lambda \ge\frac{c_0}{8}.
    \]
    In the remaining $k_1$ times we are in the bad case, where we use the first part of Lemma \ref{Lm:Barrier_improvement}. We then have the final estimate on the control
    \begin{align*}
        s_{2k_1}\geq s_{k_1}(1-C_0\lambda)^{k_1}\geq c_1,
    \end{align*}
    provided that $\lambda< 1/(2C_0)$.  
    Thus, iterating this procedure $2k_1$ times, we reach
    $$
    u_{2k_1}(x,t_{2k_1}) \geq c_1 \phi(x) \quad  x \in Q_1.
    $$
    Using one more time the first part of Lemma \ref{Lm:Barrier_improvement}, since $t_{2k_1}\leq -\frac{1}{2}$ and $\phi\geq \bar c$ in $Q_{1/2}$, for $ (x,t) \in Q_{1/2}\times [-1/2,0]$ we obtain
    \begin{align*}
            u_{2k_1}(x,t) &\geq s(t) \bar c -\frac{\eps_0}{\lambda}(t-t_{2k_1}) \geq c_1 \bar c e^{-C_0} -  \frac{\eps_0}{\lambda},
    \end{align*}
    where we used that since $t-t_{2k_1} \leq 1$ we know that $s(t)=c_1 e^{-C_0(t-t_{2k_1})} \geq c_1 e^{-C_0}$.
    
Recalling the definition of $ u_{2k_1}$ and noting that $2k_1 \leq \frac{1}{2\lambda}$, we deduce
\[
u(x,t) \geq c_1\bar c  e^{-C_0} - \frac{\eps_0}{\lambda} - 2k_1 \eps_{0}\geq c \quad  (x,t) \in Q_{1/2}\times [-1/2,0] ,
\]
by choosing  $\eps_0\leq \frac{c_1 \bar c e^{-C_0} }{4}\lambda$. This gives the desired estimate.
\end{customproof}

\section{H\"older continuity}\label{sec:holder}

Iterating Theorem \ref{Lm:Harnack}, we obtain H\"older estimates in terms of the distance $d_\lambda$, which we introduce next.
Define $d$ to be a distance which behaves as the parabolic distance in the interior and as the hyperbolic distance near the interface. Namely let 
\[ 
d((x,t),(y,s)):= \min \{|x'-y'|+|x_n-y_n|+|t-s|^{1/2}, |x'-y'|+|x_n|+|y_n|+|t-s| \}.
\]
 The metric balls associated with \(d\) are comparable, up to universal constants, with the following cylinders with centers $(y,s) \in \C_1$ and radius $r$ defined as
\begin{equation}\label{eq:def_ballBr}
    \mathcal{B}_r (y,s):= \begin{cases} Q_r(y) \times (s-r^2,s) \quad &\text{ if }r <|y_n|, \\
   Q_r(y) \times (s-r,s) \quad &\text{ if }|y_n|\leq r<2.
    \end{cases}
\end{equation}

For $\lambda \leq 1$, we define
\begin{align*}
d_\lambda((x,t),(y,s))&:= \lambda^{-1} d(\lambda(x,t), \lambda(y,s)) \\ &=\min \{|x'-y'|+|x_n-y_n|+\lambda^{-1/2}|t-s|^{1/2}, |x'-y'|+|x_n|+|y_n|+|t-s| \}.
\end{align*}
As above, we associate to the distance $d_\lambda$ the cylinders
\begin{equation} \label{eq:def_ballBrl}
    \mathcal{B}_{\lambda,r} (y,s):= \begin{cases} Q_r(y) \times (s-\lambda r^2,s) \quad &\text{ if }r <|y_n|, \\
     Q_r(y) \times (s-r,s) \quad &\text{ if }|y_n| \leq r  < 2 \lambda^{-1} .
    \end{cases}
\end{equation}
Notice that for $y_n=0$ we have $\mathcal{B}_{\lambda,r} (y,s) = \mathcal{B}_{r} (y,s)$. 
We say that $v$ is H\"older continuous with respect to $d_\lambda$ in $U$, and we write $v \in C_{d_\lambda}^{{ \beta}}(U)$, if there exists $M$ such that 
\[
\osc_{\mathcal{B}_{\lambda,r} (x,t)\cap U } v\leq M r^{{ \beta}}
\]
for any $(x,t) \in U$. Furthermore we define
\[
\|v \|_{C_{d_\lambda}^{{ \beta}}(U)} := \| v\|_{L^{\infty}(U)} + \sup_{(x,t) \neq (y,s) \in U} \frac{|v(x,t)- v(y,s)|}{d_\lambda((x,t),(y,s))^{{ \beta}}}
\]

\begin{Proposition} \label{Prop:holder} 
Let $u\in S^*_{\lambda,T}(f^{\pm})$ in $\C_1$. Then there exists a universal constant $\beta \in (0,1)$ such that $u \in C_{d_\lambda}^{\beta}(\C_{1/2})$ and
\[
\| u\|_{C_{d_\lambda}^{\beta}(\C_{1/2})} \leq C \left( \|u\|_{L^{\infty}(\C_1)} + \lambda^{-1} \|f^\pm\|_{L^{\infty}(\C_1)} \right),
\]
where $C > 0$ is a universal constant.
\end{Proposition}

\begin{proof}
    For $r \leq 1$, define the hyperbolic rescaling of $u$ 
    \[
    u_r(x,t):=u(rx,rt).
    \]
    Notice that $u_r \in S^*_{\lambda r,T}(r^2 f^\pm)$. Thus, $u_r$ satisfies the assumptions of Theorem \ref{Lm:Harnack} with $\lambda$ replaced by $\lambda r$ and $f^\pm$ replaced by $r^2 f^\pm$. Applying Theorem \ref{Lm:Harnack} to $u_r$, we obtain
    \[
    \osc_{\C_{1/2}} u_r \leq (1-c) \osc_{\C_1} u_r + C r \lambda^{-1} \| f^\pm\|_{L^\infty(\C_1)}.
    \]
Let $(y,s) \in \C_{1/2}$ and distinguish three cases based on its distance to the interface.
If \(y_n=0\), then $\mathcal{B}_{\lambda,r} (y,s) = \mathcal{B}_{r} (y,s)$ and so we obtain
    \begin{equation*}
         \osc_{\mathcal{B}_{\frac{r}{2}}(y,s)} u \leq (1-c)  \osc_{\mathcal{B}_{r}(y,s)} u + C r \lambda^{-1} \|f^\pm\|_{L^\infty(\C_1)}.
    \end{equation*}
    By a standard argument, this recurrence relation implies that for any $r < \frac{1}{2}$, we have
    \begin{equation} \label{eq:holder_interface}
         \osc_{\mathcal{B}_{r}(y,s)} u \leq C r^\beta \left( \|u\|_{L^\infty(\C_1)} + \lambda^{-1} \|f^\pm\|_{L^\infty(\C_1)} \right),
    \end{equation}

If \(0<|y_n|\le r\), then $\mathcal{B}_{\lambda,\frac{r}{8}}(y,s) \subset \mathcal{B}_{\lambda,2r}(y',0,s)$
 and \eqref{eq:holder_interface} gives
    \begin{align*}
        \osc_{\mathcal{B}_{\lambda,\frac{r}{8}}(y,s)} u &\leq \osc_{\mathcal{B}_{\lambda,2r}(y',0,s)} u  \\
        &\leq C r^\beta \left( \|u\|_{L^\infty(\C_1)} + \lambda^{-1} \|f^\pm\|_{L^\infty(\C_1)} \right),
    \end{align*}

If \(r<|y_n|\), then \(\mathcal B_{\lambda,r}(y,s)\) is contained in one
phase. Therefore we can apply standard interior H\"older estimate to
\[
v(x,t)=u(y+rx,s+\lambda r^2t),
\]
obtaining the same oscillation bound.
    
\end{proof}

\section{Comparison principle and uniqueness} \label{sec:compare}

The main result of the section is the following comparison principle for solutions to \eqref{eq:main1}.

\begin{Theorem}[Comparison principle] \label{Thm:CP}
Let $v_2$ be a viscosity supersolution to \eqref{eq:main1} and $v_1$ a viscosity subsolution to \eqref{eq:main1} such that $v_1 \leq v_2$ on $\p _p\C_1$
    then $ v_1 \leq v_2 $ in $\C_1$. 
\end{Theorem}

An immediate consequence of Theorem \ref{Thm:CP} is the uniqueness of solutions to \eqref{eq:main1} with  boundary data $g \in C(\p_p \C_1)$. 
\begin{Corollary}[Uniqueness]\label{cor:uniqueness}
    The transmission problem \eqref{eq:main1} with boundary datum $g \in  C(\p_p \C_1)$ admits at most one viscosity solution.
\end{Corollary}

In this section we will always work with the rescaling $w(x,t)=v(x,\lambda t)$ where $v$ is a solution to \eqref{eq:main1}.
We start with a preliminary result on the evolution in time of a solution with Lipschitz trace on the interface. 

\begin{Lemma} \label{lm:lipschitz_trace}
    Assume that $w\leq 1$ satisfies 
    \begin{align}\label{eq:pucci_trans}
        \begin{cases}
            \p_t w \leq \Mp(D^2w)+1 &\mbox{ in } \C_1^\pm\\
              \frac{1}{\lambda} \p_t w \leq K |\nabla' w| + K\big( (\p_n w^+)_+ + (\p_n w^-)_- \big) - K^{-1}\big( (\p_n w^+)_- + (\p_n w^-)_+ \big)& \text{ on }\C_{1}',
        \end{cases}
    \end{align}
    and 
    \[
    |\nabla'w| \leq 1 \quad \text{on }\{x_n=0 \}.
    \]
    Then 
    \[
    w(x',0,t) \geq w(x',0,0) - C \lambda^{2/3} |t|^{1/2} \quad \text{for }x' \in Q_{1/2}', \, t \in (-1,0].
    \]
\end{Lemma}

\begin{proof}
    We start by showing that if 
    \begin{equation}\label{eq:claim_initial_assump}
         w((x',0),t_0)\leq |x'|^2 \quad \text{for } t_0 <0,
    \end{equation}
    then
    \begin{equation} \label{eq:claim_tracetime}
         w(0,t)\leq C\lambda \big((t-t_0)^\frac{1}{2}+(t-t_0)\big), \quad t \in [t_0, 0]
    \end{equation}
    for some universal constant $C>0$ depending only on $n$ and $K$. For this we define
    \[
    F(x,t)=g(|x_n|,t-t_0)+M\lambda\big((t-t_0)^\frac{1}{2}+(t-t_0)\big)+|x'|^2+C(2|x_n|-x_n^2)
    \]
    where $M>0$ and $C>0$ will be chosen universally large and $g(x_n,s)$ solves 
    \begin{equation} \label{eq:construction_g}
    \begin{aligned}
        \begin{cases}
           \p_s  g=K^{-1}\p_{nn}^2 g &\mbox{ in } \R\times(0,\infty)\\
            g(x_n,0)=\chi_{(0,\infty)}-\chi_{(-\infty,0)}&\mbox{ on } \R.
        \end{cases}
    \end{aligned}
    \end{equation}
    We start by checking that $F$ is a classical supersolution to our problem in $Q_1 \times (t_0,0]$. Indeed, for $x_n\neq 0$, $t>t_0$, we have $\p_{nn}^2 g<0$ and thus
    \[
    \Mp(D^2F)=2K(n-1)+K^{-1}(\p_{nn}^2 g-2C)
    \]
    and
    \[
    \p_t F=\p_s g+M\lambda\left(\frac{1}{2}(t-t_0)^{-1/2}+1\right)
    \]
    from where we get
    \[
    \p_tF-\Mp(D^2F)=(\p_sg-K^{-1}\p^2_{nn}g)+M\lambda\left(\frac{1}{2}(t-t_0)^{-1/2}+1\right)-2K(n-1)+2C K^{-1}>1,
    \]
    provided $2CK^{-1} - 2K(n-1) > 1$.

    Now we check that the interface condition is satisfied. Note that by expressing $g$ using the fundamental solution and differentiating with respect to $x_n$ we get $\p_n g(0,t-t_0)\leq \sqrt{K/\pi} (t-t_0)^{-1/2}$. Hence we obtain
    \begin{align*}
        &\frac{1}{\lambda}\p_t F-K\big( (\p_n F^+)_+ + (\p_n F^-)_-\big)+K^{-1}\big( (\p_nF^+)_- + (\p_n F^-)_+\big)-K|\nabla' F|\\
        \geq\,&M\left(\frac{1}{2}(t-t_0)^{-1/2}+1\right)- 2K\left(\sqrt{\frac{K}{\pi}} (t-t_0)^{-1/2} + 2C\right) - 2K > 0
\end{align*}
provided we choose $M$ universally large such that $M > 4K\sqrt{K/\pi}$ and $M > 4KC + 2K$.

    To conclude, it is a simple computation to check that $F\geq w$ on the parabolic boundary for $t \in [t_0, 0]$. Hence, since $F$ is a classical supersolution, the claim \eqref{eq:claim_tracetime} follows.

    Now, notice that using the Lipschitz estimate on $\{x_n=0 \}$ at the time $t$, and for $0<r<1$ to be fixed later, we obtain
    \begin{equation} \label{eq:bound_w1}
        w(x',0,t) - w(0,t) \leq Cr^2 + \frac{|x'|^2}{r^2},
    \end{equation}
    by choosing $C$ large enough.
    This follows since the function $f(\rho)=\rho - Cr^2-\frac{\rho^2}{r^2}$ is nonpositive because it is concave in $(0,r)$, $f(0) < 0$, $f(r) <0$, and at its maximum point $\rho=\frac{1}{2}r^2$ we have $f\left(\frac{1}{2}r^2\right)=\left(\frac{1}{4}-C\right)r^2 \leq 0.$
    
    In particular, \eqref{eq:bound_w1} implies that the rescaling 
    \[
    \tilde w(y,s):=w(ry,r^2 s)-w(0,t)-Cr^2
    \]
    solves \eqref{eq:pucci_trans} with $\lambda r$ in place of $
\lambda$ and satisfies the initial growth assumption \eqref{eq:claim_initial_assump} at $s_0=t/r^2$  to apply the claim \eqref{eq:claim_tracetime}. This leads to
    $$
    \tilde w(0,s) \leq C\lambda r \big((s - s_0)^{1/2} + (s - s_0)\big), \quad \text{for } s \in [s_0, 0].
    $$
    Evaluating  this bound at $s=0$ and recalling that $s_0=t/r^2$, we obtain
    \[
  w(0,0)-w(0,t) - Cr^2 \leq C\lambda r \left( r^{-1}|t|^{1/2} + r^{-2}|t| \right),
    \]
    which simplifies to
    \[
   w(0,0)-w(0,t) \leq C(r^2+\lambda |t|^{1/2}+\lambda r^{-1} |t|).
    \]
    Choosing $r=\left( \lambda |t|\right)^{1/3}$ we obtain
    \[
    w(0,t) \geq w(0,0) - C\left( \lambda|t|^{1/2} + \left( \lambda |t|\right)^{2/3} \right) \geq w(0,0) - C \lambda^{2/3} |t|^{1/2}.
    \]
    Since $t \in (-1,0]$ was arbitrary, this gives the desired result for $x'=0$. The general case follows similarly.
\end{proof}

\begin{Remark}
    Based on the proof of Lemma \ref{lm:lipschitz_trace}, it is possible to construct a supersolution $F(x, t)$ in the $\mathcal{C}_1$ with $F(x',0, -1) = |x'|$ and $F \ge 1$ on the rest of $\partial_D \mathcal{C}_1$, such that $F(0, t) \le C \lambda^{\frac{2}{3}} |t|^{1/2}$. Furthermore, a similar argument for any ${ \alpha} > 0$ yields a supersolution satisfying $F(x', 0, -1) = |x'|^{{ \alpha}}$ and $F \ge 1$ elsewhere on $\partial_D \mathcal{C}_1$. This results in an upper bound of $F(0, t) \le C(\lambda |t|)^{\beta}$ for some parameter $\beta$ dependent on the choice of ${ \alpha}$.
\end{Remark}

We are now in position to prove Theorem \ref{Thm:CP}.
\begin{customproof}{Theorem \ref{Thm:CP}}
We start by showing that if $v_1$ and $v_2$ are respectively a subsolution and a supersolution to \eqref{eq:main1}, then $v_1$ cannot touch $v_2$ strictly by below at an interior point.

Assume by contradiction that $v_1$ touches $v_2$ strictly from below at the point $(x_0,t_0)$. This touching point  cannot be in $\C_1^- \cup \C_1^+$ according to the classical comparison principle for parabolic equations. This implies that the contact point must lie on $\{ x_n=0\}$. 

The proof consists of a series of reductions, requiring careful tracking of how every condition evolves at each step.

    The first reduction allows us to assume that $v_1$ and $v_2$ have, respectively, a semiconvex/semiconcave trace in $x'$ variable, namely \begin{equation} \label{eq:semicon}
        D_{x'}^2 v_1 \geq -I_{n-1}, \; D_{x'}^2 v_2 \leq I_{n-1}
    \end{equation}and also \begin{equation} \label{eq:smallnorm}
        \| v_i\|_{L^\infty(\C_1)} \leq 1.
    \end{equation}
    This can be done by dividing by a large constant and replacing the subsolution with its sup-convolutions in the $x'$ variable, i.e. considering

\begin{equation}\label{eq:conv_env}
v_\varepsilon(x',x_n,t)
:=
\sup_{y'\in Q'_1}
\left\{
v(y',x_n,t)-\frac{|y'-x'|^2}{2\varepsilon}
\right\}.    
\end{equation}
and the supersolution with its inf-convolution $v^\varepsilon$. Note that this smoothing preserves the sub/super solution property. Since $(v_1)_\varepsilon$ and $(v_2)^{\varepsilon}$  might not touch anymore, we translate them vertically until they touch again, at the new point
     $(y',0,s_\varepsilon) \in \C_1'$.  Note that it is possible to slightly perturb $(v_1)_\varepsilon$ and $(v_2)^{\varepsilon}$ so that they are strict sub/supersolutions respectively, by considering, for $t\leq s_\varepsilon$, 
     \begin{align*}
         &(\bar v_1)_\varepsilon = (v_1)_{\varepsilon} +\lambda\delta(t-s_\varepsilon) +K\delta( |x_n|^2 + |x_n|),\\
         &(\bar v_2)^\varepsilon = (v_2)^\varepsilon - \lambda\delta (t -s_\varepsilon) -K \delta ( |x_n|^2 + |x_n|) .
     \end{align*}
     
       Note also that $(\bar v_1)_\varepsilon (y',0,s_\varepsilon)=(\bar v_2)^{\varepsilon}(y', 0,s_\varepsilon)$ and
    \begin{equation} \label{eq:contradiction1}
     (\bar v_1)_{\varepsilon} \leq ( \bar v_2 )^{\varepsilon}+ 2\lambda\delta (t-s_\varepsilon) \; \mbox{ on }\C_1'.
       \end{equation}
Since we perturbed again $(v_1)_\varepsilon$ and $(v_2)^{\varepsilon}$, there may be more than one point where their graphs cross. We consider
        \[
    s_\varepsilon^*:= \sup \{ \, \bar t \in [-1,s_\varepsilon) : (\bar v_2)^\varepsilon (x',0,t) - (\bar v_1)_\varepsilon (x',0,t) >0 \text{ for }t \in [-1,\bar t\,),  \, x'\in Q'_1\},
    \]
    corresponding to the first time where their graphs cross at the interface $\C_1'$.
    
    To ensure that $s_{\varepsilon}^* >-1$ we choose $\delta$ small such that
    \[
    c := \inf_{x' \in Q_1'}\big\{(\bar v_2)^{\varepsilon} (x',0,-1)- (\bar v_1)_\varepsilon (x',0,-1)\big\} >0.
    \]
    We can do this
    since $v_1$ is touching $v_2$ strictly from below  at an interior point of $\C_1$ and $(v_1)_{\varepsilon} \to v_1, \; (v_2)^{\varepsilon} \to v_2$ uniformly.  
    
       At this point, it is not difficult to see that $(\bar v_1)_\varepsilon$ must touch $(\bar v_2)^\varepsilon$ strictly from below  at a point $(z',0,s_\varepsilon^*)$ in $\C_r(z',0,s_\varepsilon^*)$  for some $r$ small. In fact, if there exists a touching point $(x,t) \in \C_r(z',0,s_\varepsilon^*)$  for which this is not the case, then it is either  an interior point of $\C_r^\pm(z',0,s_\varepsilon^*)$ which contradicts parabolic comparison principle or $(x,t)=(x',0,s_0)$ with $s_0<s_\varepsilon^*$ but it would contradict the definition of $s_\varepsilon^*$. 
       
       Now, after a dilation and a translation, renaming the strict subsolution by $v_1$ and the strict supersolution by  $v_2$, we may assume that 
    \[
    v_1 (0,0)=v_2(0,0)= 0, \; v_1 < v_2 \; \text{ in } \C_1 \setminus \{(0,0)\}.
    \]

     Once again, with these hypotheses in force, we can modify $v_1,v_2$ in the interior of $\C_{1/2}^\pm$ so that they solve the equations in each side without affecting the sub/supersolution transmission condition on $\C_{1/2}'$ and \eqref{eq:semicon}, \eqref{eq:smallnorm} and \eqref{eq:contradiction1} hold on $\{ x_n=0\}$.
     To achieve this, for $i=1,2$, we define $\tilde v_i^\pm$ as 
     \begin{equation} \label{eq:tran_modified}
         \begin{aligned}
        \begin{cases}
            \p_t \tilde v_i^\pm = \tr(A^\pm(t)D^2 \tilde v_i^\pm) & \text{ in }\C_{1/2}^\pm \\
            \tilde v_i^\pm = v_i &\text{ on }\p_p \C_{1/2}^\pm
        \end{cases}
    \end{aligned}
    \end{equation}
    and call 
    \[
    \tilde v_i := \tilde v_i^+ \chi_{ \C_{1/2}^+ \cup \C_{1/2}'} + \tilde v_i^- \chi_{ \C_{1/2}^-}.
    \] 
    Notice that this modification preserves the sub/supersolution transmission condition on $\mathcal{C}_{1/2}'$. Indeed, since $v_1$ is a viscosity subsolution, the maximum principle implies $\tilde v_1^\pm \ge v_1$ in $\mathcal{C}_{1/2}^\pm$. Since $\tilde v_1$ and $v_1$ coincide on the interface $\{x_n = 0\}$, any test function $\varphi$ that touches $\tilde v_1$ from above in $\C_{1/2}'$ also touches $v_1$ from above at the same point. Because $v_1$ is a viscosity subsolution, $\varphi$ automatically satisfies the required transmission inequality. Thus, $\tilde v_1$ is a viscosity subsolution on $\mathcal{C}_{1/2}'$. A symmetric argument guarantees that $\tilde v_2$ remains a viscosity supersolution on $\mathcal{C}_{1/2}'$.
    
    We may also assume that $\nabla' \tilde v_i (0,0) = 0$ after subtracting from each $\tilde v_i$ the function $\hat a \cdot x' + b(t)$ with
    \[
    \hat a=\nabla' \tilde v_1 (0,0)=\nabla' \tilde v_2 (0,0)
    \]
    and  $\p_t b(t)=- \lambda  \hat a \cdot \hat \gamma (t), \, b(0)=0$. This change does not affect the condition \eqref{eq:contradiction1} in $\{x_n=0 \}$, however it introduces a bounded right-hand side of order $\lambda$ in \eqref{eq:tran_modified}. We may rename $\tilde v_i$ simply as $v_i$.

By \eqref{eq:semicon}, \eqref{eq:smallnorm} and \eqref{eq:contradiction1} using Lemma \ref{lm:lipschitz_trace} for any $(x',0,t) \in \P'_r$ we obtain
\begin{equation} \label{eq:est_v1_v2}
v_1 \geq  -Cr \quad \text{ and }\quad  v_2 \leq Cr \qquad \text{on } \P_r',
 \end{equation}
 where $C$ depends on the Lipschitz norm of $v_1$ and $v_2$, which may depend on $\varepsilon$ from \eqref{eq:conv_env}. From here, using the fact that $v_1\leq v_2$ we get that
  $v_1$ and $v_2$ are Lipschitz in the parabolic sense at the origin in directions $x'$ and $t$. Now we can use pointwise $C^{ \alpha}$ parabolic estimates at the origin to get, for every ${ \alpha}<1$
 \[
 \osc_{\mathcal{P}_r}v_i\leq Cr^{ \alpha} \quad \text{for all }r>0.
 \]
We iterate this argument by defining

\[
w_i:= r^{-{ \alpha}} v_i(rx,r^2 t) \quad \text{for }i=1,2,
\]
which satisfy a similar problem with $\tilde \lambda= \lambda r$, while conditions \eqref{eq:semicon}, \eqref{eq:smallnorm}, and \eqref{eq:contradiction1} continue to hold for $w_i$. Thus, applying Lemma \ref{lm:lipschitz_trace} once again, we obtain
\[
w_1(x',0,t) \geq  -Cr^{\frac 2 3} \quad \text{ and }\quad  w_2(x',0,t) \leq Cr^{ \frac 2 3} \qquad \text{for } x \in Q_{1/2}'.
\]
Rewriting these bounds in terms of $v_1$ and $v_2$, we improve upon the estimates in \eqref{eq:est_v1_v2} to yield
\begin{equation} \label{eq:est_v1_v2_2}
    v_1 \geq  -Cr^{{ \alpha}+ \frac 2 3} \quad \text{ and }\quad  v_2 \leq Cr^{{ \alpha}+ \frac 2 3} \qquad \text{on } \P_r'.
\end{equation}
This, in turn, implies $C^{{ \alpha}+\frac{2}{3}}$ regularity in the parabolic sense at the origin in every variable.

Next, we modify $v_i$ again by subtracting the functions $\p_n  v_i^+ (0) (x_n)_+ + \p_n  v_i^- (0) (x_n)_- + b_i(t)$, where $b_i(t)$ is chosen so that $\p_t b_i(t)=  \lambda \left(\gamma^+ \p_n v_i^+ (0) - \gamma^{-} \p_n  v_i^-(0) \right)$ and $b_i(0)=0$. In this way, the newly defined functions remain respectively a subsolution and a supersolution with a bounded right-hand side of order $\lambda$. 
Furthermore, since $v_1$ touches $v_2$ from below at the origin, we have $\p_n v_1^+(0) - \p_n v_2^+(0) \leq 0$ and $\p_n v_1^-(0) - \p_n v_2^-(0) \geq 0$. This ensures that the modification does not violate condition \eqref{eq:contradiction1}. Indeed, by applying \eqref{eq:contradiction1} to $v_1$ and $v_2$, we get on $\{x_n = 0\}$
\begin{align*}
    v_1(x',0,t) &- v_2(x',0,t) - (b_1(t) - b_2(t)) = v_1(x',0,t) - v_2(x',0,t) + \int_{t}^{0} \p_\tau ( b_1 -b_2)(\tau) \, d\tau \\
    &\leq 2\lambda \delta t +\lambda \int_{t}^{0} \left( \gamma^+(\tau) (\p_n v_1^+(0) - \p_n v_2^+(0)) -\gamma^{-}(\tau) (\p_n v_1^-(0) - \p_n v_2^-(0))\right) \, d\tau \leq 2\lambda\delta t. 
\end{align*}
Repeating the previous argument shows that \eqref{eq:est_v1_v2_2} actually holds with $r^{{ \alpha}+\frac{4}{3}}$ in place of $r^{{ \alpha}+\frac{2}{3}}$. Because ${ \alpha} + \frac{4}{3}>2$, it follows that $v_1(0,t)\geq -C|t|^{1+\beta}$ and $v_2(0,t)\leq C|t|^{1+\beta}$ for all sufficiently small $t<0$ and some $\beta>0$, giving a contradiction with \eqref{eq:contradiction1}.
\end{customproof}

Since both the operator and the transmission condition are linear, we can obtain the following result making use of Theorem \ref{Thm:CP}.

\begin{Proposition} \label{prop:diff_sol}
        If $u$ and $v$ are viscosity solutions to \eqref{eq:main1}, then so is $u-v$.
\end{Proposition}

\begin{proof}
    We  show that $u-v$ is a viscosity subsolution in the sense of Definition \ref{def:visc}. The supersolution case follows similarly. Consider a test function $\varphi$ which touches from above $u-v$ at a point $(x_0,t_0) \in \C_1$. We may assume the touch is strict after replacing $\varphi$ with $\varphi + (t-t_0)^2 + |x-x_0|^4$. Either $(x_0,t_0) \in \C_1^\pm$ or $(x_0,t_0) \in \C_1'$. We limit ourselves to consider the latter case which is more delicate. If $(x_0,t_0) \in \C_1^\pm$ a simpler standard viscosity argument applies. Let  $(x_0,t_0) \in \C_1'$ and $\varphi \in C_x^{2}(\overline{ \C_\delta^\pm(x_0,t_0)}) \cap  C_t^{1}(\overline{ \C_\delta^\pm(x_0,t_0)})$.
    Assume by contradiction that 
    \begin{equation}\label{eq:contr_1}
            \p_t\varphi(x_0,t_0)  +\hat{\gamma}(t_0) \cdot  \nabla' \varphi(x_0,t_0)  - \gamma^+( t_0) \p_n \varphi^+(x_0,t_0) + \gamma^-(t_0) \p_n \varphi^-(x_0,t_0) >0 . 
    \end{equation}
    At this point, we can further modify the test function to ensure that it acts as a strict supersolution in the interior domains as well. Let us replace $\varphi$ with
    \[
    \tilde{\varphi}(x,t) := \varphi(x,t) + \eta |x_n| - C x_n^2 \quad \text{ in } \C_r(x_0,t_0) ,
    \]
    where $\eta,r > 0$ are small parameters and $C > 0$ is a large constant. By choosing $r \leq \delta$ such that $\eta |x_n| - C x_n^2\geq 0$ in $Q_{r}(x_0)$  we have $\tilde{\varphi} \ge \varphi \ge u-v$ in $\C_{r}(x_0,t_0)$. Thus, $\tilde{\varphi}$ still touches $u-v$ strictly from above at $(x_0,t_0)$. 
    Furthermore, for $x_n \neq 0$, we have  $D^2 \tilde{\varphi} = D^2 \varphi - 2C (e_n \otimes e_n)$. Exploiting the uniform ellipticity $A^\pm_{nn} \ge K^{-1} > 0$, by choosing $C$  large we guarantee that
    \begin{equation} \label{eq:contr_interior}
    \lambda \p_t \tilde{\varphi} - \tr(A^\pm(t) D^2 \tilde{\varphi}) = \lambda \p_t \varphi - \tr(A^\pm(t) D^2 \varphi) + 2C A^\pm_{nn}(t) > 0 \quad \text{in } \C_r^\pm(x_0,t_0).
    \end{equation}
    On the interface $\{x_n = 0\}$, in view of \eqref{eq:contr_1} choosing $\eta$ sufficiently small, the transmission operator yields
    \begin{align*}
            &\p_t\tilde{\varphi}(x_0,t_0)   \!+\!\hat{\gamma}(t_0)  \cdot  \nabla' \tilde{\varphi}(x_0,t_0)  \!-\! \gamma^+(t_0) \p_n \tilde{\varphi}^+(x_0,t_0) \!+\! \gamma^-(t_0) \p_n \tilde{\varphi}^-(x_0,t_0) \\ 
            &=  \p_t\varphi(x_0,t_0)  \!+\!\hat{\gamma}(t_0) \!\cdot\!  \nabla' \varphi(x_0,t_0)  \!-\! \gamma^+\!( t_0) \p_n \varphi^+(x_0,t_0) \!+\! \gamma^-\!(t_0) \p_n \varphi^-\!(x_0,t_0)   \!-\! \eta(\gamma^+\!(t_0) \!+\! \gamma^-\!(t_0))>0.
    \end{align*}
Hence, up to renaming $\tilde{\varphi}$ back to $\varphi$, we can assume without loss of generality that $\varphi$ satisfies both \eqref{eq:contr_1} and the strict interior inequality \eqref{eq:contr_interior}.
    After reducing $r$, if necessary, by the continuity of the derivatives of $\varphi$ and of the coefficients $\hat{\gamma}, \gamma^\pm, A^\pm$, the strict inequality \eqref{eq:contr_1} holds for all points $(x,t) \in \C_r'(x_0,t_0)$ and  \eqref{eq:contr_interior} holds in $\C_r^\pm(x_0,t_0)$.

We now show that $v+\varphi$ is a strict viscosity supersolution  in $\C_r(x_0,t_0)$. Let $(x_1,t_1) \in \C_r(x_0,t_0)$ and let $\psi$ touch $v+\varphi$ at $(x_1,t_1)$ from below. Then $\psi -\varphi$ touches $v$ at $(x_1,t_1)$ from below. 
    If $(x_1,t_1) \in \C_r'(x_0,t_0)$, the supersolution property of $v$ yields
    \begin{align*}
          \p_t(\psi - \varphi)\!(x_1,t_1)  \!+ \!\hat{\gamma}(t_1)\!\cdot \! \nabla' (\psi - \varphi)\!(x_1,t_1)  \!- \! \gamma^+\!( t_1) \p_n (\psi - \varphi)^+\!(x_1,t_1) \!+ \! \gamma^-\!(t_1) \p_n (\psi - \varphi)^-\!(x_1,t_1) \geq 0.
    \end{align*}
    This together with \eqref{eq:contr_1} evaluated at $(x_1,t_1)$ implies that the transmission operator on $\psi$ is strictly positive. Similarly, if $(x_1,t_1) \in \C_r^\pm(x_0,t_0)$, evaluating the interior operator on $\psi - \varphi$ and adding \eqref{eq:contr_interior} ensures that $\lambda \p_t \psi(x_1,t_1) - \tr(A^\pm D^2\psi)(x_1,t_1) > 0$.

Thus, $v + \varphi$ is a strict supersolution in $\C_r(x_0,t_0)$ that touches $u$ strictly from above at $(x_0,t_0)$. Hence, since $(x_0,t_0)$ is a strict maximum for $u-v-\varphi$, there exists $\sigma > 0$ such that $u \leq v + \varphi - \sigma$ on $\p_p \C_{r}(x_0,t_0)$. By Theorem \ref{Thm:CP}, we deduce $u \leq v + \varphi - \sigma$ everywhere in $\C_{r}(x_0,t_0)$, which gives a contradiction at $(x_0,t_0)$ since $(u-v)(x_0,t_0)=\varphi(x_0,t_0)$.

\end{proof}

\section{Existence of viscosity solutions}
\label{sec:existence}

The goal of this section is to prove existence of viscosity solutions to \eqref{eq:main1} with continuous boundary data $g \in C(\partial_p \mathcal{C}_1)$ via Perron's method.

\begin{Theorem}[Existence]
\label{thm:existence}
Let $g \in C(\partial_p \mathcal{C}_1)$. There exists a unique viscosity solution $u \in C(\overline{\mathcal{C}_1})$ to \eqref{eq:main1} satisfying $u = g$ on $\partial_p \mathcal{C}_1$.
\end{Theorem}

First, we construct suitable global barriers to ensure that the Perron class is nonempty.

\begin{Lemma}
\label{lem:barriers_existence}
Let $g$ be a smooth function on $\partial_p \mathcal{C}_1$. There exist a viscosity subsolution $\underline{u} \in C^2(\mathcal{C}_1^\pm) \cap C(\overline{\mathcal{C}_1}) $ and a viscosity supersolution $\overline{u} \in C^2(\mathcal{C}_1^\pm) \cap C(\overline{\mathcal{C}_1})$ to \eqref{eq:main1} such that 
\begin{equation}
    \begin{cases}
        \underline{u} \leq \overline{u} \quad &\text{ in } \C_1, \\ 
        \underline{u} \le g \le \overline{u} \quad &\text{ on } \partial_p \mathcal{C}_1.
    \end{cases}
\end{equation}

\end{Lemma}

\begin{proof}
Let $\underline{\psi} \in C(\overline{\mathcal{C}_1})$ be the unique viscosity solution to 
\begin{equation} \label{eq:psi_under}
\begin{cases}
\lambda \partial_t \underline{\psi} - \mathcal{M}_K^-(D^2 \underline{\psi}) = 0 & \text{in } \mathcal{C}_1, \\
\underline{\psi} = g  & \text{on } \partial_p \mathcal{C}_1.
\end{cases}
\end{equation}
The existence, uniqueness and interior $C^{2,\alpha}$ regularity of $\underline{\psi}$ and the Lipschitz continuity both in space and time up to the boundary follow from the results in \cite{wang1992regularityI,wang1992regularityII,lian2026boundary}.  Hence, the derivatives $\partial_t \underline{\psi}$ and $\nabla \underline{\psi}$ are bounded on $\mathcal{C}_1'$ by a constant $C_0 > 0$. 

Define $\underline{u} \in C(\overline{\mathcal{C}_1})$ by setting
\[
\underline{u}(x,t) := \underline{\psi}(x,t) + C_1 (|x_n| - 1),
\]
where $C_1> 0$ is a constant to be chosen depending on $C_0$.
Since $C_1(|x_n|-1) \le 0$ in $\mathcal{C}_1$, we have $\underline{u} \le \underline{\psi} = g$ on $\partial_p \mathcal{C}_1$.
Notice that $\underline{u} \in C^2(\C_1^\pm)$ and thus we can verify directly the subsolution condition in $\mathcal{C}_1^\pm$. Indeed we have
\[
\lambda \partial_t \underline{u} - \operatorname{tr}(A^\pm(t) D^2 \underline{u}) = \lambda \partial_t \underline{\psi} - \operatorname{tr}(A^\pm(t) D^2 \underline{\psi})\leq 0  \quad \text{ in }\C_1^\pm, 
\] 
since $\underline{\psi}$ solves \eqref{eq:psi_under}.
 Thus, $\underline{u}$ is a classical subsolution in $\mathcal{C}_1^\pm$.
We now check that $\underline{u}$ satisfies the subsolution condition on $\mathcal{C}_1'$. From the expression of $\underline{u}$ we infer 
\[
\partial_t \underline{u} = \partial_t \underline{\psi}, \quad \nabla' \underline{u} = \nabla' \underline{\psi}, \quad \text{and} \quad \partial_n \underline{u}^\pm = \partial_n \underline{\psi} \pm C_1 \quad \text{on } \mathcal{C}_1'.
\]
Thus the transmission operator on $\C_1'$ yields
\begin{align*}
    &\partial_t \underline{u} + \hat{\gamma}(t) \cdot \nabla' \underline{u} - \gamma^+(t) \partial_n \underline{u}^+ + \gamma^-(t) \partial_n \underline{u}^- \\
&= \partial_t \underline{\psi} + \hat{\gamma}(t) \cdot \nabla' \underline{\psi} - (\gamma^+(t) - \gamma^-(t))\partial_n \underline{\psi} - C_1(\gamma^+(t) + \gamma^-(t))
\\ &\leq C_0(1 + 3K) - C_1(\gamma^+(t) + \gamma^-(t)),
\end{align*}
where $C_0$ is the universal bound on the derivatives of $\underline{\psi}$, and $\hat \gamma$, $\gamma^+$ and $\gamma^-$ are bounded by $K$.
Moreover using that
  $\gamma^\pm(t) \ge K^{-1}$, we can choose $C_1$ sufficiently large such that
\[
\partial_t \underline{u} + \hat{\gamma}(t) \cdot \nabla' \underline{u} - \gamma^+(t) \partial_n \underline{u}^+ + \gamma^-(t) \partial_n \underline{u}^-\le C_0(1 + 3K) - 2C_1K^{-1} \le 0 \quad \text{on } \mathcal{C}_1'.
\]
 This shows that $\underline{u}$ is a viscosity subsolution on $\mathcal{C}_1'$.
The supersolution case can be handled in a similar way defining $\overline{u}$ as
$$
\overline{u}(x,t) := \overline{\psi}(x,t) - C_1 (|x_n| - 1)  \in C^2(\C_1^\pm) \cap C(\overline{\mathcal{C}_1}).
$$ 
In this case, $\overline{\psi}$ is the unique solution of the problem $\lambda \partial_t \overline{\psi} - \mathcal{M}_K^+(D^2 \overline{\psi}) = 0$ with boundary datum
$g$ on $\p_p \C_1$. In this way $ g \leq \overline{u} $ on $\p_p \C_1$ and as a consequence of the Comparison Principle (Theorem \ref{Thm:CP}) we deduce $\underline{u}\leq \overline{u}$ in $\C_1$.

\end{proof}

Given a smooth function $g$ on $\p_p \C_1$, we define the Perron class associated with $g$ as
\begin{equation} \label{def:S_g}
    \mathcal{S}_g := \{ v \in USC(\overline{\mathcal{C}_1}) \mid v \text{ is a viscosity subsolution to \eqref{eq:main1} and } v \le g \text{ on } \partial_p \mathcal{C}_1 \}.
\end{equation}
By Lemma \ref{lem:barriers_existence}, $\underline{u} \in \mathcal{S}_g$, so $\mathcal{S}_g \neq \emptyset$. Furthermore, by the Comparison Principle (Theorem \ref{Thm:CP}), every $v \in \mathcal{S}_g$ satisfies $v \le \overline{u}$ in $\overline{\mathcal{C}_1}$.

 We can therefore define the Perron solution to \eqref{eq:main1} with boundary datum $g$  as
\[
u(x,t) := \sup_{v \in \mathcal{S}_g} v(x,t) \quad \text{for } (x,t) \in \overline{\mathcal{C}_1},
\]
and its upper and lower semicontinuous envelopes as
\[
u^*(x,t) := \limsup_{(y,s) \to (x,t)} u(y,s), \quad u_*(x,t) := \liminf_{(y,s) \to (x,t)} u(y,s).
\]
By construction, $\underline{u} \le u_* \le u \le u^* \le \overline{u}$ in $\overline{\mathcal{C}_1}$. 

In the following lemma, by constructing suitable local barriers, we show that $u_* = u^* = g$ on $\partial_p \mathcal{C}_1$.

\begin{Lemma} \label{lm:Local_barrier}
    Let $g$ be a smooth function on $\p_p \C_1$. For any $(x_0,t_0) \in \p_p \C_1$ there exists a viscosity subsolution $\underline{h}$ and a viscosity supersolution $\overline{h}$  to \eqref{eq:main1} such that $\underline{h}\leq g \leq \overline{h}$ on $\p_p \C_1$ and such that $\underline{h}(x_0,t_0)=g(x_0,t_0)=\overline{h}(x_0,t_0)$.
\end{Lemma}

\begin{proof}
We first focus on showing the existence of $\underline{h}$.
    Let $(x_0,t_0) \in \p_p \C_1$. Let $\underline{\psi}$ be the function given by \eqref{eq:psi_under}, which satisfies $\underline{\psi}(x_0,t_0) = g(x_0,t_0)$. We construct a local subsolution of the form
    $$ \underline{\tilde{h}}(x,t) := \underline{\psi}(x,t) - M w(x,t), $$
    where $M > 0$ is a large constant and $w(x,t) \ge 0$ is a supersolution to \eqref{eq:main1} in the set
    \[
    \mathcal{N}_r(x_0, t_0) := (x_0 + Q_r) \times (t_0 - r^2, t_0 + r^2) \cap \C_1,
    \]
    for some small $r$, such that $w(x_0, t_0) = 0$ and $w > 0$ on $\partial \mathcal{N}_r(x_0, t_0) \cap \C_1$. 
    For the construction of $w$ we distinguish three different cases.

   \underline{\textit{Case 1: $(x_0)_n \neq 0$ and $t_0>-1$.}} For $r$ small, $\mathcal{N}_r(x_0, t_0)$ does not intersect $\C_1'$. Let $\nu$ be the inward unit normal to the face of the cube $Q_1$ containing $x_0$. We define the barrier
        $$ w(x,t) := 1 - e^{-C(x-x_0)\cdot \nu} + |x-x_0|^2 + (t - t_0)^2. $$
        Since $x \in \overline{Q_1}$, we have $(x-x_0)\cdot \nu \ge 0$, which ensures $w \ge 0$. The strict inequality $w > 0$ for $(x,t) \neq (x_0, t_0)$ is guaranteed by the positive quadratic terms.
        For the supersolution condition in $\mathcal{N}_{r}(x_0,t_0)$ we have
        $$ \lambda \p_t w - \Mp(D^2 w) = 2\lambda(t-t_0) +  K^{-1} C^2 e^{-C(x-x_0)\cdot \nu} - 2nK \geq 0, $$
       which holds by taking $C > 0$ sufficiently large and $r$ small enough so that the term $2\lambda(t-t_0)$ is negligible.
 
    \underline{\textit{Case 2: $(x_0)_n = 0$ and $t_0 >-1$. }} 
    In this case we borrow a spatial barrier constructed in \cite{erneta2026existence}.
     Up to translations and rotations, we can assume the exterior spatial normal at $x_0$ is $-e_1$.
    We split the spatial coordinates as $x = (x_1, x'', x_n) \in \R \times \R^{n-2} \times \R$ and introduce the polar coordinates $(\rho, \theta)$ in the $(x_1, x_n)$-plane centered  at $((x_0)_1, 0)$. That is, $x_1 - (x_0)_1 = \rho \cos \theta$ and $x_n = \rho \sin \theta$  with $\theta \in [-\pi/2, \pi/2]$.
    Following the construction in \cite[Lemma 5.2]{erneta2026existence}, based on earlier works of Lieberman \cite{lieberman1986mixed} and Miller \cite{Miller}, for points in $\mathcal{N}_r(x_0,t_0)$ we define 
    $$ w(\rho,\theta,x'',t) := \rho^\beta v(\theta) + c_0|x'' - x''_0|^2 + (t - t_0)^2, $$
    where $v(\theta) = 2 - e^{2K^2(|\theta| - \frac{\pi}{2})} \in (1,2)$, $\beta=(e^{K^2 \pi}-\frac{1}{2})^{-1}\in (0,1)$, and $c_0=\frac{2^\beta}{n}e^{-K^2\pi}$. This choice makes $w(x_0,t_0)=0$ and $w(x,t)>0$ for $(x,t) \neq (x_0,t_0)$.
 By computing explicitly the interior equations we get
 \begin{align*}
      \lambda \p_t w - \tr(A^\pm(t)D^2 w ) &\geq  2\lambda(t-t_0) - \Mp(D^2 w ) \\ &\geq 2\lambda(t-t_0) + 2 \rho^{\beta-2}Ke^{-K^2\pi} -c_0 \Mp(D^2 (|x''-x_0''|^2))  \\ 
      & \geq 2\lambda(t-t_0) + 2 \rho^{\beta-2}Ke^{-K^2\pi} - 2K(n-2)c_0 \geq \rho^{\beta-2}Ke^{-K^2\pi} > 0
      \end{align*}
      for $\rho <r$ and $r$ chosen small enough (since $\beta < 1$, the term $\rho^{\beta-2}$ dominates the small time derivative). On the interface $\theta = 0$, evaluating the transmission operator we obtain 
      \begin{align*}
          &\partial_t w + \hat{\gamma}(t) \cdot \nabla' w- \gamma^+(t) \partial_n w^+ + \gamma^-(t) \partial_n w^-  \\
          & \geq  2(t-t_0) - \hat\gamma(t) \cdot \left( \beta(2-e^{-K^2\pi})\rho^{\beta-1},2c_0(x''-x_0'')\right) + (\gamma^+(t) + \gamma^-(t)) 2K^2e^{-K^2\pi} \rho^{\beta-1}\\
          &\geq -2r^2 -2Ke^{-K^2\pi} \rho^{\beta-1} - 2 K c_0 r + 4K e^{-K^2\pi} \rho^{\beta-1} \geq Ke^{-K^2\pi} \rho^{\beta-1} >0
      \end{align*}
  where the term $\rho^{\beta-1}$ absorbs the terms depending on $r$ since $\beta <1$. Hence, $w$ is a strict supersolution in $\mathcal{N}_r(x_0,t_0)$.
  
      \underline{\textit{Case 3: $t_0 =-1$. }} 
      Here, the initial point lies at the bottom of the cylinder, so we only consider the forward-in-time neighborhood $\mathcal{V}_r(x_0, -1) := (x_0 + Q_r) \times [-1, -1+r^2)$. In this case, it is enough to set
    $$ w(x,t) := C(t+1) +|x - x_0|^2. $$
     Since $\partial_t w = C> 0$, taking $C$ large guarantees that $w$ is a strict supersolution in $\mathcal{V}_r(x_0, -1) \cap \C_1$ and $w(x_0,-1)=0$.

\medskip
Having constructed the barriers in all cases, we now proceed to define the local subsolution. Let $\mathcal{W}_r(x_0,t_0)$ denote $\mathcal{N}_r(x_0, t_0)$ if $t_0 > -1$, and $\mathcal{V}_r(x_0, -1)$ if $t_0 = -1$.
For  $M > 0$ large enough, the function $\underline{\tilde{h}} := \underline{\psi} - Mw$ is a strict subsolution in $\mathcal{W}_r(x_0,t_0) \cap \mathcal{C}_1$. Furthermore, since $w(x,t) > 0$ for $(x,t)\neq (x_0,t_0)$, we can choose $M$ sufficiently large to ensure that $\underline{\tilde{h}} \le \underline{u}$  on $\partial \mathcal{W}_r(x_0,t_0) \cap \mathcal{C}_1$.
Thus defining
$$
\underline{h}(x,t) := \begin{cases} \max\{\underline{u}(x,t), \underline{\tilde{h}}(x,t)\} & \text{in } \overline{\mathcal{W}_r}(x_0,t_0) \cap \mathcal{C}_1, \\ \underline{u}(x,t) & \text{in } \overline{\mathcal{C}_1} \setminus \mathcal{W}_r(x_0,t_0), \end{cases}
$$
yields a continuous viscosity subsolution $\underline{h}$ with $\underline{h}\leq g$ on $\partial_p \mathcal{C}_1$ and such that $\underline{h}(x_0,t_0)=g(x_0,t_0)$.

In order to construct the upper barrier $\overline{h}$, we consider $\overline{\psi}$ to be the supersolution from Lemma \ref{lem:barriers_existence} satisfying $\overline{\psi}(x_0,t_0) = g(x_0,t_0)$. In this case, $\overline{\tilde{h}} = \overline{\psi} + Mw$ is a strict supersolution in $\mathcal{W}_r(x_0,t_0) \cap \mathcal{C}_1$, where $w \ge 0$ is the exact same function constructed above. For $M>0$ large enough, we have $\overline{\tilde{h}} \ge \overline{u}$ on $\partial \mathcal{W}_r(x_0,t_0) \cap \mathcal{C}_1$. Thus, defining
$$
\overline{h}(x,t) := \begin{cases} \min\{\overline{u}(x,t), \overline{\tilde{h}}(x,t)\} & \text{in } \overline{\mathcal{W}_r}(x_0,t_0) \cap \mathcal{C}_1, \\ \overline{u}(x,t) & \text{in } \overline{\mathcal{C}_1} \setminus \mathcal{W}_r(x_0,t_0), \end{cases}
$$
gives a viscosity supersolution in $\mathcal{C}_1$ satisfying $\overline{h} \ge g$ on $\partial_p \mathcal{C}_1$ and $\overline{h}(x_0,t_0) = g(x_0,t_0)$. This concludes the proof.

\end{proof}

As an immediate consequence of Lemma \ref{lm:Local_barrier}, for any $(x_0,t_0)$, the function $\underline{h} \in \mathcal{S}_g$ such that $\underline{h}(x_0,t_0) = g(x_0,t_0)$. By definition of $u$ as the supremum over $\mathcal{S}_g$, we have $u \ge \underline{h}$ in $\overline{\C_1}$. 
Conversely, considering $\overline{h}$ given by Lemma \ref{lm:Local_barrier} from the Comparison Principle (Theorem \ref{Thm:CP}), we get that any $v \in \mathcal{S}_g$ satisfies $v \le \overline{h}$ in $\overline{\C_1}$, which implies $u \le \overline{h}$. 
Taking the liminf and limsup as $(x,t) \to (x_0, t_0)$, the continuity of the barriers yields
\[
g(x_0, t_0) = \underline{h}(x_0, t_0) \le u_*(x_0,t_0) \le u^*(x_0,t_0) \le \overline{h}(x_0, t_0) = g(x_0, t_0).
\]
Hence, it holds $u_* = u^* = g$ on $\p_p \C_1$.

Now we provide the proof of Theorem \ref{thm:existence}.

\begin{customproof}{Theorem \ref{thm:existence}} ~

\underline{\textit{Step 1: Existence for smooth boundary data.}}
We begin by assuming that $g$ is a smooth function on $\p_p \C_1$ and we consider the Perron class $\mathcal{S}_g$ associated with $g$, as defined in \eqref{def:S_g}.

By standard viscosity theory (see  \cite[Lemma 4.18]{jesus2026fully}),  $u^*$ is a viscosity subsolution to \eqref{eq:main1} in $\mathcal{C}_1$. 

We now prove that $u_*$ is a viscosity supersolution to \eqref{eq:main1} in $\mathcal{C}_1$. 
Assume by contradiction that there exists $\varphi$ defined in $\C_1$ which touches $u_*$ strictly from below at $(x_0,t_0)$ in $\C_r(x_0,t_0)$ but
\begin{equation*}
 \begin{aligned}
     & \lambda \p_t \varphi(x_0,t_0) - \tr(A^{\pm}(t_0) D^2 \varphi(x_0,t_0)) < 0, & \quad \text{if }(x_0,t_0) \in \C_{1}^\pm,  \\
        &\p_t \varphi(x_0,t_0) < -\hat{\gamma}(t_0) \cdot  \nabla' \varphi(x_0,t_0)  + \gamma^+( t_0) \p_n \varphi^+(x_0,t_0) - \gamma^-(t_0) \p_n \varphi^-(x_0,t_0),  & \quad \text{if }(x_0,t_0) \in \C_{1}'.
 \end{aligned}
\end{equation*}

If $(x_0,t_0)\in \C_1'$, we can replace $\varphi$ with $\varphi-\eta|x_n| + C x_n^2$, with $\eta$ small and $C$ large, as in the proof of Proposition \ref{prop:diff_sol}, and assume without loss of generality that $\varphi$
is also a strict subsolution in $\mathcal{C}_r^\pm(x_0,t_0)$ while preserving the strict inequality on the interface.

We now define the perturbed test function
\[
\varphi_{\delta,\theta}(x,t) := \varphi(x,t) + \delta - \theta \left( |x - x_0|^2 + (t - t_0)^2 \right),
\]
where $\theta, \delta > 0$ are small constants. At $(x_0,t_0)$, the first derivatives of $\varphi$ coincide with the ones of $\varphi_{\delta,\theta}$.  
Choosing $\delta < \theta r^2/2$, we obtain that $u_*-\varphi_{\delta,\theta}>c$ on $\partial_p \C_r(x_0,t_0)$, for some small $c>0$. Hence, by lower semi-continuity of $u_*-\varphi_{\delta,\theta}$, there exists a small $\varepsilon>0$ such that $u_*-\varphi_{\delta,\theta}>c/2$ on $\partial_p \C_r(x_0,t_0+\varepsilon)$. By continuity, $\varphi_{\delta,\theta}$ remains a strict subsolution in a small neighborhood $\C_r(x_0,t_0+\eps)$, by choosing $\theta > 0$ and $\eps$  sufficiently small.

Thus, the function
\[
w(x,t) := 
\begin{cases}
\max\{u^*(x,t), \varphi_{\delta,\theta}(x,t)\} & \text{in } \overline{\C_r}(x_0,t_0+\eps), \\
u^*(x,t) & \text{in } \overline{\mathcal{C}_1} \setminus \C_r(x_0,t_0+\eps)
\end{cases}
\]
is upper semicontinuous by construction. 
Because $u^*$ is a viscosity subsolution, $w$ is also a viscosity subsolution and $w\leq g$ on $\p_p \C_1$ and thus belongs to the Perron class $\mathcal{S}_g$. As $u$ is the supremum over $\mathcal{S}_g$, we must have $w \le u$.

However, since $u_*$ is the lower semicontinuous envelope of $u$, there exists a sequence $(x_k, t_k) \to (x_0,t_0)$ such that $u(x_k, t_k) \to u_*(x_0,t_0)$. Evaluating the inequality $w \le u$ along this sequence yields
\[
\varphi_{\delta,\theta}(x_k, t_k) \le \max\{u^*(x_k, t_k), \varphi_{\delta,\theta}(x_k, t_k)\} = w(x_k, t_k) \le u(x_k, t_k).
\]
Passing to the limit as $k \to \infty$, since $\varphi_{\delta,\theta}$ is continuous, we get
\[
\varphi(x_0,t_0) + \delta = \varphi_{\delta,\theta}(x_0, t_0) \le u_*(x_0, t_0).
\]
This creates a contradiction with $\varphi(x_0,t_0) = u_*(x_0,t_0)$.

Hence, $u_*$ is a viscosity supersolution to \eqref{eq:main1} in $\C_1$.
 By Theorem \ref{Thm:CP}, we obtain $u^* \le u_*$ in $\overline{\mathcal{C}_1}$. Since $u_* \le u \le u^*$ holds everywhere, this implies that $u = u^* = u_*$ in $\overline{\mathcal{C}_1}$. 
This proves the existence of a unique viscosity solution $u \in C(\overline{\mathcal{C}_1})$ with smooth boundary datum $g$.

\underline{\textit{Step 2: Existence for continuous boundary data.}}
Now, let $g \in C(\partial_p \mathcal{C}_1)$ be a continuous boundary datum. We can find a sequence of smooth functions $\{g_k\}_{k\geq 0}$ on $\partial_p \mathcal{C}_1$ such that $g_k \to g$ uniformly on $\partial_p \mathcal{C}_1$. 
By the proof above, for each $k$, there exists a unique continuous viscosity solution $u_k \in C(\overline{\mathcal{C}_1})$ satisfying $u_k = g_k$ on $\partial_p \mathcal{C}_1$.

Let $M = \|g_k - g_m\|_{L^\infty(\partial_p \mathcal{C}_1)}$. Notice that $v = u_m + M$ is a viscosity solution to \eqref{eq:main1}. 
On the parabolic boundary $\partial_p \mathcal{C}_1$, we have
\[
u_k = g_k \le g_m + M = u_m + M = v.
\]
By Theorem \ref{Thm:CP}, it follows that $u_k \le u_m + M$ in $\overline{\mathcal{C}_1}$. Reversing the roles of $k$ and $m$, we obtain $u_m \le u_k + M$. Thus,
\[
\|u_k - u_m\|_{L^\infty(\overline{\mathcal{C}_1})} \le \|g_k - g_m\|_{L^\infty(\partial_p \mathcal{C}_1)}.
\]
Since $\{g_k\}_{k \geq 0}$ is a Cauchy sequence with respect to the uniform norm on the boundary, the above inequality implies that $\{u_k\}_{k \geq 0}$ is a Cauchy sequence in $L^\infty(\overline{\mathcal{C}_1})$. Consequently, $u_k$ converges uniformly to a continuous function $u \in C(\overline{\mathcal{C}_1})$. By stability properties of viscosity solutions under uniform limits, see, for instance, \cite[Proposition 2.9]{CC}, $u$ is a viscosity solution to \eqref{eq:main1}. Finally, since $u_k = g_k$ on $\partial_p \mathcal{C}_1$, passing to the uniform limit yields $u = g$ on $\partial_p \mathcal{C}_1$, completing the proof.
\end{customproof}

\section{Proof of Theorem \ref{Thm:C1gamma}}\label{sec:C1gamma}

In this section, we investigate the regularity of viscosity solutions to \eqref{eq:main1} up to and across the interface $\C_1'$. We present hereafter the proof of Theorem \ref{Thm:C1gamma}.

The proof proceeds in two stages. First, the H\"older estimate from Proposition \ref{Prop:holder} gives enough tangential regularity to reduce the problem, after freezing \(x'\), to a one-dimensional transmission problem with a mildly singular right-hand side. The one-dimensional estimate yields \(C^{1,\alpha}\) expansions from both sides of the interface. In the second stage, the improved regularity along the normal direction allows us to bootstrap \(C^{2,\alpha}\) estimates for the
one-dimensional problem and hence to classical
regularity.

\begin{customproof}{Theorem \ref{Thm:C1gamma}} 
The existence and uniqueness of a viscosity solution $u$ with boundary datum $g \in C(\p_p \C_1)$ is ensured by Corollary \ref{cor:uniqueness} and Theorem \ref{thm:existence}. 

In the following we show that $u$ is a classical solution to \eqref{eq:main1} and satisfies \eqref{eq:C11} and \eqref{eq:transmission_1}.

Let $(y,s) \in \C_{1/2}^{\pm}$ and let $r=|y_n|$. From Proposition \ref{Prop:holder} we deduce that
$$
\osc_{\mathcal{B}_{\lambda,r}(y,s)} u \leq C r^{{ \beta}},$$
for some ${\beta}>0$. The balls $\mathcal{B}_{\lambda,r}$ were defined in \eqref{eq:def_ballBrl}. Consider the following rescaling of $u$
$$
\tilde u(x,t):=u(y+rx,s+r^2\lambda t),
$$
that satisfies
$$
\p_t \tilde{u} = \tr \left( \overline{A}(t) D^2 \tilde{u} \right) \quad \text{ in } \C_1,
$$
where $\overline{A}(t)=A^\pm(s+r^2\lambda t)$, and the sign $\pm$ is chosen according to the phase containing the point $(y,s)$. From the assumptions on $A^{\pm}$ we know $|\p_t \overline{A}(t)| \leq C$ and thus applying interior parabolic estimates to $\tilde u$, we get
\begin{equation} \label{eq:int_est_on_u}     \begin{aligned}     \vert{}\p_n u(y,s)\vert{} &= r^{-1}\vert{}\p_n \tilde{u}(0,0)\vert{} \leq Cr^{-1}\osc_{\C_1} \tilde u \\     &= Cr^{-1}\osc_{\mathcal{B}_{\lambda,r}(y,s)} u \leq Cr^{{ \beta}-1}.     
\end{aligned} \end{equation}

Since the coefficients $A^\pm(t)$, $\gamma^\pm(t)$, and $\hat{\gamma}(t)$ depend only on the time variable, both the interior equation and the transmission condition in \eqref{eq:main1} are invariant under spatial translations in the tangential directions $x'$. 
As a consequence, using the linearity of the problem, for any tangential vector $e=(\hat e,0) \in \R^{n-1}\times \R$ and any $h>0$, the difference quotient 
$$ D_h^e u(x,t) := \frac{u(x+he, t) - u(x,t)}{h^\beta} $$
is still a viscosity solution to \eqref{eq:main1} by Proposition \ref{prop:diff_sol}. 
Applying Proposition \ref{Prop:holder} to $D_h^e u$, with estimates independent on $h$, and iterating finitely many times on nested cylinders, we may pass to the limit as $h \to 0$, as in \cite[Corollary 5.7]{CC}. This yields the corresponding estimates for the higher order partial derivatives with respect to $x'$. Using \eqref{eq:int_est_on_u}, for $\alpha \leq \beta$ we find 
\begin{align} \label{eq:mixed_bounds}
    |D_{x'}^2 u| \leq C, \quad |\nabla' \p_n u| \leq C|x_n|^{{ \alpha}-1} \quad \text{in }\C_{1/2}^{\pm}.
\end{align}

We now reduce our problem to a one-dimensional problem in space freezing the $x'$ variable. For simplicity we can take $x'=0$ and, looking at the problem in the $x_n,t$ variables solved by the function $w(x_n,t):=u(0,x_n,t)$, we find
\begin{equation} \label{eq:1d_transmission}
    \begin{cases}
        \lambda \p_t w = q^{\pm}(t) \p_{nn}^2w + \tilde{h}^{\pm}(x_n,t) \quad &\text{ in }\C_{1}^{\pm} \left(\subset \mathbb{R}^2 \right), \\
        \p_t w=\gamma^{+}( t) \p_{n}w^+ - \gamma^{-}( t) \p_n w^-  + \xi(t)  \quad &\text{ on }\C_{1}' \left( := \{x_n=0 \} \times (-1,0] \right) , \\
    \end{cases}
\end{equation}
with
\begin{equation}
    \begin{aligned} &q^{\pm}(t):=A^\pm_{nn}(t), \quad K^{-1} \leq q^{\pm}(t),\gamma^\pm(t) \leq K, \quad  |\p_t q^{\pm}(t)| \leq \lambda^{-1}  \\
        &\tilde{h}^\pm(x_n,t):= \sum_{(i,j) \neq (n,n)} A_{ij}^{\pm}(t) \p_{ij}^2 u(0,x_n,t), \quad |\tilde{h}^\pm(x_n,t)| \leq K |x_n|^{{ \alpha}-1}, \\
        &\xi(t):= - \hat\gamma(t)\cdot \nabla'u(0,0,t), \quad |\xi(t)| \leq CK,
    \end{aligned}
\end{equation}
 where the bounds follow from assumptions \eqref{Assump1}.

The transmission condition in \eqref{eq:1d_transmission} is satisfied in the viscosity sense. 
Indeed, if we consider $\varphi=\varphi(x_n,t) \in C_{x_n}^{1}(\overline{ \mathcal{B}_\delta^+(0,0)}) \cap  C_t^{1}(\overline{ \mathcal{B}_\delta^+(0,0)})$ and $\varphi \in C_{x_n}^{1}(\overline{ \mathcal{B}_\delta^-(0,0)}) \cap  C_t^{1}(\overline{ \mathcal{B}_\delta^-(0,0)})$  touching $w$ from above (resp. from below) at $(0,0) \in \R^2$, then the function
    \[
\varphi(x_n,t)+ \nabla' u(0,0,0) \cdot x' \pm C|(x,t)|^{1+\alpha/2}
    \]
touches $u$ from above (resp. below) at $(0,0,0)$ in $ \mathcal{B}_{\delta} \subset \R^{n+1}$. Here $|(x,t)|$ denotes the Euclidean norm in $\R^{n+1}$.
This is a consequence of the interior parabolic H\"older continuity of $\nabla' u$, which implies, for some $y'$ in the line segment $[0,x']$,
\begin{equation}\label{eq:est_on_x_prime}
\begin{aligned}
     \left| u(x,t) - u(0,x_n,t)-\nabla'u(0,0,0)\cdot x'\right| \,&\leq \left| \nabla' u(y',x_n,t) -\nabla'u(0,0,0) \right||x' | \\
    &\leq C\left(|(y',x_n)|^{\beta} +|t|^{\beta/2}\right)|x'|
    \leq C|(x,t)|^{1+\alpha},
\end{aligned}
\end{equation}
for $\alpha\leq \beta/2$.
We can eliminate the term $\xi(t)$ from the transmission condition by considering $w(x_n,t) - \int_{0}^t \xi(s) \, ds$. For convenience, we then work with the time-rescaled function $v(x_n,t) := w(x_n, \lambda t)$, which solves \eqref{eq:1d_transmission_redu} with $h^\pm(x_n,t) := \tilde{h}^\pm(x_n, \lambda t) - \lambda \xi(\lambda t)$.
  Hence, after relabeling $\lambda t$ with $t$ in the definition of the functions $h^{\pm},q^{\pm}$ and $\gamma^\pm$, we can reduce to consider the following problem   
\begin{equation} \label{eq:1d_transmission_redu}
    \begin{cases}
        \p_t v = q^{\pm}(t) \p_{nn}^2 v + h^{\pm}(x_n,t) \quad &\text{ in } (0,1)^{\pm} \times (-\lambda^{-1},0] , \\
        \p_t v= \lambda \left( \gamma^{+}( t) \p_n v^+ - \gamma^{-}( t) \p_n v^- \right)  \quad &\text{ on } \{ 0 \} \times  (-\lambda^{-1},0] , \\
    \end{cases}
\end{equation}
with the bounds
\begin{equation} \label{eq_bounds_on_h}
        \begin{aligned}
            K^{-1} \leq q^{\pm},\gamma^{\pm} \leq K, \quad |\p_t q^{\pm}| \leq K, \quad |\p_t \gamma^\pm| \leq 1, \quad |h^\pm|\leq K|x_n|^{{ \alpha}-1}.
        \end{aligned}
    \end{equation}
The $C^{1,{ \alpha}}$ estimates for the one-dimensional problem \eqref{eq:1d_transmission_redu} are obtained in Lemma \ref{lm:reg1d} which in terms of $u$ gives
\begin{equation} \label{eq:est_on_w}
    \begin{aligned}
    &\left| u(0,x_n,t) - \left( u(0,0,t) + \p_n u^{+}(0,0,t) (x_n)_+ -  \p_n u^{-}(0,0,t) (x_n)_- \right)\right| \leq C |x_n|^{1+{ \alpha}} \quad &\text{in }\C_{1/2}^{\pm}, \\
     &\left|\p_n u^{\pm}(0,x_n,t)- \p_n u^{\pm}(0,y_n,s)\right| \leq C \left( |x_n-y_n|^{{ \alpha}}+ \lambda^{-{ \alpha}/2} |t-s|^{{ \alpha}/2}\right) \quad &\text{in }\C_{1/2}^{\pm}.
\end{aligned}
\end{equation}

   This together with \eqref{eq:est_on_x_prime} leads to 
   \[
   |u(x,t)-\left( u(0,t) + \p_n u^+(0,t) (x_n)_{+} - \p_n u^-(0,t) (x_n)_{-}  + \nabla'u(0,0) \cdot x'\right)| \leq C r^{1+{ \alpha}} \quad \text{in }\C_{r}^{\pm},
   \]
   which we write in the compact form
   \[
   |u-\ell_{\tilde a^\pm,\tilde b}| \leq C r^{1+{\alpha}} \quad \text{in }    \C_{r}^{\pm}.
   \]
   where $\ell_{\tilde a^\pm, \tilde b}=\tilde a^{\pm}(t) \cdot x + \tilde b(t)$
   with
   \[
   \tilde a^\pm(t)=(\Hat{a},\tilde a_n^\pm(t)):=(\nabla' u(0,0),\p_n u^\pm(0,t) ) ,\quad \tilde b(t):=u(0,t)
   \]
   and the following relation descending from the transmission condition

   \[
 \p_t \tilde b(t)= -\hat{\gamma}(t) \cdot \nabla'u(0,t) + \gamma^+(t) \tilde a_n^+(t) - \gamma^-(t) \tilde a_n^-(t).
   \]

Now we slightly modify the coefficients $\tilde a^\pm(t)=(\Hat{a},\tilde a_n^\pm(t))$ and $\tilde b(t)$ to  $a^\pm(t)=(\Hat{a}, a_n^\pm(t))$ and $b(t)$ so that they satisfy \eqref{eq:transmission_1} and 
    \[
    \| u - \ell_{a^{\pm},b} \|_{L^\infty( \C^{\pm}_r)} \leq C r^{1+{ \alpha}},
    \]
    with $|\hat a| \leq C$, $|\p_t a_n^{\pm}| \leq C r^{{ \alpha}-2} \lambda^{-1}$.
By \eqref{eq:est_on_w} we know that 
\[
|\tilde a_n^\pm(t)-\tilde a_n^\pm(s)|\leq C \lambda^{-\alpha/2}|t-s|^{\alpha/2},
\]
which ensures that within any time interval of length $\lambda r^2$, the variation of $\tilde a_n^\pm$ is at most $Cr^{ \alpha}$.
To smooth out these coefficients uniformly over the entire interval $[-r,0]$, we perform a time-averaging procedure. Let $\eta$ be a smooth mollifier compactly supported in $[-1,1]$, and let $\eta_{\lambda r^2}$ be its standard rescaling. Extending $\tilde a_n^\pm$ constantly for both $t \geq 0$ and $t \leq -r$, we define the regularized coefficients as
    \[
    a_n^\pm(t) := (\tilde a_n^\pm \ast \eta_{\lambda r^2})(t), \quad \text{and we set} \quad \hat a := \nabla'u(0,0).
    \]
Since the mollification window has size $\lambda r^2$, and the variation of $\tilde a_n^\pm$ is uniformly bounded by $C r^\alpha$ on any such window, this guarantees that for all $t \in [-r,0]$ we have both $|\tilde a_n^\pm(t) - a_n^\pm(t)| \leq Cr^{ \alpha}$  and 
\begin{equation} \label{bound_on_an'}
\begin{aligned}
   |\p_t a_n^\pm(t)| &\leq   \sup_{|\tau| \le \lambda r^2} |\tilde a_n^\pm(t-\tau) - \tilde a_n^\pm(t)|  \cdot \int_{-\lambda r^2}^{\lambda r^2} |\p_\tau \eta_{\lambda r^2}(\tau)| \, d\tau \\
   &\leq C (\lambda r^2)^{-1}\sup_{|\tau|\le \lambda r^2} |\tilde{a}_n^\pm(t-\tau) - \tilde{a}_n^\pm(t)|\leq  C \lambda^{-1} r^{\alpha - 2}.
\end{aligned}
\end{equation}

The regularity of $u$ in the $x'$-direction leads to
\begin{align} \label{eq:reg_u_x'}
    |\p_t \tilde b(t) +\hat{\gamma}(t) \cdot \hat{a} - \gamma^+(t) \tilde a_n^+(t) + \gamma^-(t) \tilde a_n^-(t)| \leq |\hat \gamma \cdot (\nabla' u(0,t)- \nabla' u(0,0))|\leq C|t|^{\alpha}.
\end{align}

We now define $b(t)$ for $t\leq 0$ by setting $b(0)=\tilde b(0)$ and 
\[
 \p_t b(t) = -\hat{\gamma}(t) \cdot \hat{a} + \gamma^+(t) a_n^+(t) - \gamma^-(t) a_n^-(t).
\]
Furthermore, combining the uniform bound $|\tilde a_n^\pm(t) - a_n^\pm(t)| \leq Cr^{ \alpha}$ valid for $t \in [-r,0]$ and \eqref{eq:reg_u_x'}, we deduce 
that $|\p_t (\tilde b - b)(t)|\leq Cr^\alpha$ for all $t \in [-r, 0]$.
By integrating this inequality over the entire interval $t \in [-r, 0]$, we easily deduce $|b(t) - \tilde b(t)| \leq Cr^{1+{ \alpha}}$. Replacing $\ell_{\tilde a^\pm, \tilde b}$ with the updated affine function $\ell_{a^\pm, b}$ absorbs these differences, yielding the desired approximation \eqref{eq:transmission_1}.

While this already yields pointwise spatial $C^{1,\alpha}$  at each fixed time up to each side of the interface, the bound $|\partial_t a_n^\pm| \le C \lambda^{-1} r^{\alpha-2}$ is not enough to  control the slopes of the approximating planes which can tilt and oscillate wildly as $r \to 0$. To overcome this problem we need to bootstrap higher regularity.

For this, we make use of the readily established regularity in the $e_n$ direction. Since any tangential derivative $\partial_{x_i} u$  ($i<n$) solves the same transmission problem, we can apply the $C^{1,\alpha}$ estimates directly to $\partial_{x_i} u$. This implies that its full spatial gradient is $C^\alpha$. In particular, both the purely tangential derivatives $\partial_{ij}^2 u$ and the mixed derivatives $\partial_{in}^2 u$ are H\"older continuous, satisfying
\begin{align} \label{eq:mixedbound2}
 |\partial_{in}^2 u (0,x_n,\lambda t) - \partial_{in}^2 u (0,y_n,\lambda s)| \leq  C\left(|x_n-y_n|^\alpha + |t-s|^{\alpha/2}\right).
\end{align}
This in turn improves the regularity of $h^\pm$ in \eqref{eq_bounds_on_h}.
Indeed, we can show that 
\begin{equation}\label{eq:bound_h111}
    [h^\pm]_{C^\alpha\left((0,3/4)^\pm \times (-\frac{9}{16}\lambda^{-1},0] \right)} \leq C.
\end{equation}
Recall that in the rescaled variables, 
$$
h^\pm(x_n, t) = \sum_{(i,j) \neq (n,n)} A_{ij}^\pm(\lambda t) \partial_{ij}^2 u(0,x_n,\lambda t) + \lambda \hat{\gamma}(\lambda t) \cdot \nabla' u(0,0,\lambda t).
$$
Regarding the coefficients, assumptions \eqref{Assump1} ensure that  $A_{ij}^\pm(\lambda t)$ and $\hat{\gamma}(\lambda t)$ are Lipschitz continuous. 
Since $\nabla' u$ is already known to be regular, $h^\pm$ is a finite sum of products of bounded, H\"older continuous functions. This implies \eqref{eq:bound_h111}, which allows us to apply the second part of Lemma \ref{lm:reg1d} and obtain that 
$$
\|v\|_{C^{2,\alpha}\left( (0,1/\sqrt{2})^\pm \times (-(2\lambda)^{-1},0] \right)} \leq C.
$$
This in particular implies that 
$$|\p_{nn}^2 v(x_n, t)| \leq C \quad \text{ in } (0,1/\sqrt{2})^\pm \times (-(2\lambda)^{-1},0],$$
which gives $|\p_{nn}^2 u(x',x_n,t)| \leq C$ in $\C_{1/2}^\pm$ by repeating the same argument for different $x'$. Combining this with \eqref{eq:mixed_bounds} and  \eqref{eq:mixedbound2} leads to \eqref{eq:C11}.

Lemma \ref{lm:reg1d} also implies that $\p_{nn}^2 u$ and $\p_t u$ are H\"older continuous. This establishes that $u$ is indeed a classical solution to \eqref{eq:main1}. As a byproduct also the regularity of $a_n^\pm$ improves and repeating the time-averaging procedure implies that \eqref{bound_on_an'} holds with $\alpha+1$ instead of $\alpha$ thus giving $ |\p_t a_n^\pm| \leq C \lambda^{-1} r^{\alpha-1}$.

 \end{customproof}

\section{Estimates for the 1D case}\label{sec:1dcase}
We are interested in studying the regularity of the transmission problem \eqref{eq:1d_transmission_redu},  which reduces to a one-dimensional spatial problem in the  variables $x,t$ in $\mathbb{R}^2$. In this section we denote any derivative in space and time of a function $v$ respectively with $\p_x v$ (or $\p_{xx}^2 v$) and $\p_t v$. 

In this setting the parabolic cylinders are defined as
\[
\P_r:= (-r,r) \times (-r^2,0], \quad \P_r^{\pm}:= (0,r)^\pm \times (-r^2,0], 
\]
 where we used the notation $(0,r)^{\pm}$ to denote $(0,r)$ or $(-r,0)$ respectively. Here, for a nonnegative integer $k$, we denote by $C_{x,t}^{k,{ \alpha}}$ the H\"older spaces induced by the parabolic distance.
 
 We recall the scaling properties of \eqref{eq:1d_transmission_redu}. If $w$ solves \eqref{eq:1d_transmission_redu} in $\P_\rho$ then the rescaling
\[
\tilde w (x,t):= \rho^{-\sigma} w(\rho x, \rho^2 t)
\]
solves \eqref{eq:1d_transmission_redu} in $\P_1$ with coefficients
\begin{equation} \label{eq:coeff_rescale}
    \tilde q^{\pm}(t):= q^{\pm}(\rho^2 t), \quad \tilde h^{\pm}(x,t)= \rho^{2-\sigma} h^{\pm}(\rho x, \rho^2 t), \quad \tilde \lambda = \rho \lambda, \quad \tilde\gamma^{\pm}(t) = \gamma^{\pm}(\rho^2 t).
\end{equation}
As long as $\sigma \leq 1+{ \alpha}$, the hypotheses on the coefficients are preserved and $\tilde \lambda \to 0$ as $\rho \to 0$.

We start with a preliminary result about the H\"older regularity of solutions to \eqref{eq:1d_transmission_redu}.
\begin{Lemma} \label{lm:holder_reg1d}
    Let $v$ be a viscosity solution to \eqref{eq:1d_transmission_redu} in $\P_1$ with $\lambda \leq 1$ and the coefficients satisfy the bounds in \eqref{eq_bounds_on_h}.
    Then 
    \begin{equation} \label{eq:holder1d}
        \| v\|_{C_{x,t}^{0,\beta}\left( \overline{\P}_{1/2}^{\pm}\right)} \leq C \left( \| v\|_{L^{\infty}(\P_1)} +1 \right),
    \end{equation}
    for some $\beta >0$
    \end{Lemma}
\begin{proof} After an initial dilation, we may assume that $\| v\|_{L^{\infty}(\P_1)}\leq 1$ and $\lambda \leq \lambda_0$ is small. 
 For parabolic cylinders centered inside $\P_1^{\pm}$ we obtain H\"older estimates directly from regularity of parabolic equations.
 We concentrate our analysis to show a diminishing oscillation property for $v$ for points on the axis $\{x=0\}$.
  Without loss of generality we concentrate on the regularity around the point $(0,0)$.
  More precisely, it is enough to prove that for a small universal $\rho$ we have $\osc_{\P_\rho} v \leq \frac{3}{2}$.
  
Consider the function $g$, introduced in Lemma \ref{lm:lipschitz_trace} to be the function solving \eqref{eq:construction_g}. We then define
    \[
    F(x,t):=C_1g(|x|,t+1)+ \frac{1}{4}(t+1)^{1/2}-C_2|x|^{1+{ \alpha}}.
    \]
     By choosing appropriately the constants $C_1,C_2$ we want to show that $F$ is a supersolution to \eqref{eq:1d_transmission_redu} that lies above $v$ on $\p_p \P_1$. Thus, by Theorem \ref{Thm:CP}, we deduce that $v \leq F$ in $\P_1$  and $F(0,0)=\frac{1}{4}$. Hence if $v(0,-1) \leq 0$ then $v \leq 1/2$ in $\P_\rho$ with $\rho=c_0$ sufficiently small, implying $\osc_{\P_\rho} v \leq \frac{3}{2}$. 
    Expressing  explicitly $g$ as
    \[
    g(x,t)=(4K^{-1}\pi t)^{-1/2} \int_{\mathbb{R}} e^{-K|x-y|^{2}/4t} g(y,0) \, dy,
    \]
    we deduce the following
    \begin{equation} \label{eq:properties_g}
    g(0,t)=0 \; \quad \p_x g(0,t)\leq C_3 t^{-1/2} \; \text{for }t>0, \text{ and}  \quad \p_t g \leq 0 \; \text{for }x>0.
    \end{equation}
    Using the last of these properties, in $\P_1^\pm$ we have
\begin{align*}
        \p_t F&=K^{-1} \p_{xx}^2 F+ { \alpha}(1+{ \alpha})K^{-1}C_2|x|^{{ \alpha}-1} + \frac{1}{8}(t+1)^{-1/2} \\
        &\geq q^{\pm}(t)  \p_{xx}^2 F + h^{\pm}(x,t)
    \end{align*}
   by choosing $C_2$ sufficiently large.
On $\{ x=0\}$, the first two properties in \eqref{eq:properties_g}, together with  $\lambda \leq \lambda_0$ small, imply that 
\begin{align*}
    \p_t F(0,t)&=\frac{1}{8}(t+1)^{-1/2} \geq 2 \lambda K C_1 C_3 (t+1)^{-1/2} \\
    &\geq \lambda \left( \gamma^{+}(t) \p_x F^+(0,t) - \gamma^{-}(t) \p_x F^-(0,t) \right).
\end{align*}
Furthermore, by choosing $C_1>C_2$ sufficiently large we can make $F>1\geq v$ on $\p_p \P$ and obtain the desired improvement of oscillation.  

\end{proof}

The next result deals with higher regularity.

\begin{Lemma} \label{lm:reg1d}
    Let $v$ be a viscosity solution to \eqref{eq:1d_transmission_redu} in $\P_1$ with $\lambda \leq 1$ and the coefficients satisfy the bounds in \eqref{eq_bounds_on_h}.
    Then 
    \begin{equation} \label{eq:c1a_1d}
        \| v\|_{C_{x,t}^{1,{ \alpha}}\left( \overline{\P}_{1/2}^{\pm}\right)} \leq C \left( \| v\|_{L^{\infty}(\P_1)} +1 \right),
    \end{equation}
    and the transmission condition is satisfied in a classical sense.
    Moreover, if in addition we have 
    \[
    \| h^\pm \|_{C_{x,t}^{0,{ \alpha}}(\P_{3/4})} +  \| \gamma^\pm \|_{C_{t}^{0,{ \alpha/2}}(\P_{3/4})}   \leq K,   \]
    then
       \begin{equation} \label{eq:c2a_1d}
        \| v\|_{C_{x,t}^{2,{ \alpha}}\left( \overline{\P}_{1/2}^{\pm}\right)} \leq C \left( \| v\|_{L^{\infty}(\P_1)} +1 \right).
    \end{equation}

\end{Lemma}
\begin{proof} We split the proof in three steps. 

\smallskip
\underline{ \it Step 1: $C^{1,{ \alpha}}$ estimates.} In order to achieve \eqref{eq:c1a_1d}, we establish a $C^{1,{ \alpha}}$ estimate at the origin from both sides. After an initial dilation and up to dividing by a large constant, we may assume that $\lambda \leq \delta$, $|h^\pm| \leq \delta |x|^{{ \alpha}-1}$ and $\|v \|_{L^{\infty}(\P_1)} \leq \delta$, for $\delta$ small to be fixed later.
We start with the following one-step improvement claim.

\textit{Claim.} Assume that there exist affine functions
$$
\ell^\pm_k(x, t) = a^\pm_k x + b_k(t)
$$
satisfying the transmission condition
\begin{equation} \label{eq:transmi1}
\p_t b_k(t) = \lambda(\gamma^+(t)a^+_k - \gamma^-(t)a_k^-) \quad  \text{with $|a^\pm_k| \le A_0$.}
\end{equation}
Suppose that $\ell^\pm_k$ approximate $v$ in $\P^\pm_\rho$ up to order $1+{ \alpha}$, namely
\[\|v^\pm - \ell^\pm_k\|_{L^\infty(\P^\pm_\rho)} \le \rho^{1+{ \alpha}}, \quad \text{for some } \rho \le 1.
\]
Then, provided $\delta > 0$ is sufficiently small, there exist new affine functions $\ell^\pm_{k+1}(x, t) = a^\pm_{k+1} x + b_{k+1}(t)$ whose coefficients satisfy \eqref{eq:transmi1}, such that
\[
\|v^\pm - \ell^\pm_{k+1}\|_{L^\infty(\P^\pm_{c_1\rho})} \le (c_1\rho)^{1+{ \alpha}},
\]
and $|a^\pm_{k+1} - a^\pm_k| \le C\rho^{ \alpha}$, where $c_1 \in (0, 1)$, $A_0 \ge 1$ and $C > 0$ are universal constants.

Notice that once this claim is proven, the desired $C^{1,{ \alpha}}$ estimate at the origin from both sides follows by iterating it indefinitely for $\rho_k = \rho c_1^k$. The base case at $\rho_0 = 1$ holds starting with $\ell^\pm_0 \equiv 0$, since our initial normalization guarantees $\|v\|_{L^\infty(\P_1)} \le \delta \le \rho_0^{1+\alpha}$.

\textit{Proof of the Claim.} We prove this one-step improvement by a compactness argument. 

It is not difficult to see that, since 
\[
\p_t \ell_k^\pm = \p_t b_k(t) = \lambda \left( \gamma^{+}(t) a_k^+ - \gamma^{-}(t) a_k^- \right) \leq 2\lambda K A_0 \leq C\delta,
\]
the function $v^\pm - \ell_k^\pm$ solves \eqref{eq:1d_transmission_redu} with a modified $h^{\pm}$ satisfying $|h^{\pm}| \leq \delta |x|^{{ \alpha}-1}+ C\delta$. This implies that the rescaled error function
\begin{equation} \label{eq:def_v_tilde}
\tilde v(x,t) = \rho^{-(1+{ \alpha})} (v^\pm-\ell_k^\pm)(\rho x,\rho^2 t)
\end{equation}
solves  \eqref{eq:1d_transmission_redu} with coefficients as in \eqref{eq:coeff_rescale}. Notice that $\tilde v$ is continuous across $\{x=0 \}$ and since $\| \tilde  v\|_{L^{\infty}(\P_1)} \leq 1$ then, by Lemma \ref{lm:holder_reg1d}, we also know that 
\[
 \| \tilde v\|_{C_{x,t}^{0,\beta}(\P_{\frac{1}{2}})} \leq C.
\]
Arguing by compactness, we know that for every $\varepsilon_1>0$ there exists a small $\delta_1>0$ such that if $\delta<\delta_1$, then there exists a function $\bar v$, solution to
\begin{equation} \label{eq:1d_transmission_redu_limit}
    \begin{cases}
\p_t \overline{v} = \overline{q}^{\pm} \p_{xx}^2 \overline{v} \quad &\text{ in } \P_{1/2}^\pm, \\
        \p_t \overline{v}= 0  \quad &\text{ on } \{x= 0 \},
    \end{cases}
\end{equation}
with $\overline{q}^\pm=q^\pm(0) \in (K^{-1},K)$, and satisfying
\[
\|\tilde v-\bar v\|_{L^\infty(\P_{1/2})}\leq \varepsilon_1.
\]
Since $\overline{v}$ is constant on $\{x=0 \}$, one can apply $C^2$ parabolic estimates in both sides and obtain
\[
|\overline{v}^{\pm} - (\overline{a}^\pm x + \overline{b}) |\leq C \tau^{2} \leq \frac{1}{8} \tau^{1+{ \alpha}}    \quad \text{ in } \P_{\tau}^{\pm},
\]
for $\tau \leq c_1$ small universal. Combining both estimates and choosing $\varepsilon_1=\frac{1}{8} \tau^{1+{ \alpha}}$, which also fixes $\delta_1$ universally, we obtain
\[
|\tilde v^\pm - (\overline{a}^\pm x + \overline{b}) | \leq \frac{1}{4} \tau^{1+{ \alpha}}    \quad \text{ in } \P_{\tau}^{\pm}.
\]
Rescaling back to $v$ and using the definition of $b_k(t)$, we get 
\[
|v^{\pm} - \left(a_{k+1}^\pm x +b(t) \right)| \leq \frac{1}{2} (\tau \rho)^{1+{ \alpha}}    \quad \text{ in } \P_{\rho\tau}^{\pm}, 
\]
with
\[
a_{k+1}^\pm := a_k^\pm + \rho^{{ \alpha}} \overline{a}^\pm, \quad b(t) := b_k(t)+ \rho^{1+{ \alpha}} \overline{b}.
\]
We now slightly perturb $b$ so that the approximating functions satisfy the transmission condition. For this, we define $b_{k+1}(t)$ such that
\[
\p_t b_{k+1}(t)=\lambda \left( \gamma^+(t) a_{k+1}^+ - \gamma^-(t) a_{k+1}^-\right), \quad b_{k+1}(0)=b(0).
\]
Defining $\ell_{k+1}^\pm(x,t):=a_{k+1}^\pm x + b_{k+1}(t)$ we check that it satisfies the hypotheses of the claim. In fact,
by definition of $a_{k+1}^\pm$ we have $|a_{k+1}^\pm - a_k^\pm| \leq  C \rho^{ \alpha}$ and 
\begin{align*}
    |\p_t (b_{k+1}-b)(t)|&= \lambda |\gamma^+(t)(a_{k+1}^+ - a_k^+) - \gamma^-(t)(a_{k+1}^- - a_k^-)| \\
    &\leq 2 \lambda K \rho^{ \alpha} |\overline{a}| \leq C\lambda \rho^{{ \alpha}},
\end{align*}
implying that for $t \in (-(\tau\rho)^2,0]$ 
\[
|(b_{k+1}-b)(t)| \leq \|\p_t (b_{k+1}-b)\|_{L^\infty} |t| \leq C \lambda \rho^{{ \alpha}} (\tau \rho)^{2} \leq \frac{1}{2} (\tau \rho)^{1+{ \alpha}}
\] 
where $\rho \leq \delta$ is sufficiently small. 
We can now conclude that $\ell_{k+1}^\pm$ approximates $v$ from both sides at the origin with order $1+{ \alpha}$ in $\P_{\tau \rho}$, namely
\[
|v^\pm - \ell_{k+1}^\pm | \leq |v^\pm -(a_{k+1}^\pm x +b)| + |b-b_{k+1}| \leq (\tau \rho)^{1+{ \alpha}} \quad \text{in }\P_{\tau \rho}^{\pm},
\]
which implies the desired claim. The same reasoning can be applied to any point on $\{ x= 0\} $.

For points in the interior of $\P_{1/2}^\pm$, notice that the bound on $h^\pm$ implies that for all $r \leq 1$,
\[ \fint_{\C_{r}} |h^{\pm}| \, dx \leq C r^{{ \alpha}-1},
\]
which is the condition on the right-hand side appearing in \cite[Theorem 1.3]{wang1992regularityII} to get $C^{1,{ \alpha}}$.

\smallskip
\underline{ \it Step 2: Regularity on the interface.} $\;$ 
We prove that the interface condition in \eqref{eq:1d_transmission_redu}  holds in the classical sense.
Assume by contradiction that there exist a sequence $t_k \to 0^-$ and an $\eta > 0$ such that 
\begin{equation} \label{eq:case_1_d}
    \frac{v(0,t_k)-v(0,0)}{t_k} < \mu:= \lambda \big(\gamma^+(0)\partial_x v^+(0,0) - \gamma^-(0)\partial_x v^-(0,0)\big) - 2K\eta.
\end{equation}

We construct a strict comparison subsolution $g(x,t)$ in a neighborhood $D_k = [-c(\eta), c(\eta)] \times [t_k, 0]$. Consider the function $g$ to be defined as
\[
g(x,t) := v(0,0) + \mu t + 
\begin{cases}
x\left(\partial_x v^+(0,0) - \frac{\eta}{2}\right) + C x^{1+\alpha} & \text{if } x \ge 0 \\
x\left(\partial_x v^-(0,0) + \frac{\eta}{2}\right) + C |x|^{1+\alpha} & \text{if } x < 0
\end{cases}
\]

Choosing $C$ large, it is not difficult to impose that the function $g$ is a strict classical subsolution in $D_k$ in $x \neq 0 $ using the bound on $q^\pm\geq K^{-1}$ given in \eqref{eq_bounds_on_h}. On $\{x=0\}$ our choice of $\mu$ (which includes the $-2K\eta$ penalization to absorb coefficient oscillations), yields, after evaluating the fluxes,
$$
\partial_t g(0,t) = \mu < \lambda \big( \gamma^+(t) \partial_x g^+(0, t) - \gamma^-(t) \partial_x g^-(0, t) \big),
$$
satisfying the strict boundary condition.

We now translate $g$ vertically upwards to find the first contact point with $v$ in $D_k$. 
By choosing the domain width $c(\eta)$ sufficiently small, the strict inequalities $\partial_x g^+ < \partial_x v^+$ and $\partial_x g^- > \partial_x v^-$ ensure that the spatial minimum of $v-g$ occurs at $\{x=0\}$. Therefore, the contact point cannot lie in the interior. 
Furthermore, the contact cannot occur at $t=t_k$ because $v(0, t_k) - g(0, t_k) > 0$ by our initial contradiction hypothesis. 
Consequently, the first touch must occur at the interface $(0, t^*)$ for some $t^* \in (t_k, 0]$. This violates the definition of a viscosity solution, since $v$ cannot be touched from below by a strict subsolution on the interface.

The case when \eqref{eq:case_1_d} is reversed can be treated similarly. 

\smallskip
\underline{ \it Step 3: $C^{2,{ \alpha}}$ estimates.} 
On the hyperplane $\{x=0\}$, we have $\p_x v^\pm, \gamma^{\pm} \in C_t^{0,{\alpha}/2}$, and the interface condition implies that $v(0,t) \in C_t^{1,{\alpha}/2}$. Consequently, we can apply standard $C_{x,t}^{2,{\alpha}}$ Schauder estimates for the heat equation up to the interface from both sides.

\end{proof}

\section*{Acknowledgements}

The authors would like to thank Daniela De Silva and Ovidiu Savin for fruitful
conversations on the topic of this paper. 
D.G. and D.J. wish to thank the Department of Mathematics of Columbia University for the warm hospitality.

\section*{Fundings}
This research is partially supported by University of Bologna funds
\say{Attività di collaborazione con università del Nord America 2024 MAT} in the framework of the project: \say{Interplaying problems in analysis and geometry}. 

\bibliographystyle{abbrv}
\bibliography{bibStefan}
\Addresses
\end{document}